\documentclass[12pt,oneside,onecolumn,open]{article}
\usepackage[latin1]{inputenc}
\usepackage[english]{babel}
\usepackage{amsmath}
\usepackage{amsthm}
\usepackage{amsfonts}
\usepackage{amssymb}
\usepackage{graphicx}
\usepackage{hyperref}
\hypersetup{
    colorlinks=true,
    linkcolor=blue,
    filecolor=magenta,      
    urlcolor=cyan,
    pdftitle={Overleaf Example},
    pdfpagemode=FullScreen,
    }

\usepackage[left=2.5cm,right=2.5cm,top=3cm,bottom=3cm]{geometry}
\usepackage[usenames]{color}
\usepackage[titletoc,toc,page]{appendix}

\newtheorem{teo}{Theorem}[section]
\newtheorem{defi}[teo]{Definition}

\newtheorem{propo}[teo]{Proposition}
\newtheorem{obs}[teo]{Remark}
\newtheorem{cor}[teo]{Corollary}
\newtheorem{lema}[teo]{Lemma}
\usepackage{authblk}
\author[1]{Miguel Ballesteros\footnote{miguel.ballesteros@iimas.unam.mx}}
\author[2]{Gerardo Franco C\'ordova}
\author[1]{Jonathan Gil}
\author[1]{Ivan Naumkin}
\affil[1]{Instituto de Investigaciones en Matem\'aticas Aplicadas y en Sistemas (IIMAS), Universidad Nacional Aut\'onoma de M\'exico (UNAM),  M\'exico.}
\affil[2]{Institute of Analysis and Algebra, Technische Universit\"at Braunschweig,  Germany.}

\title{   On Matrix Valued Schr\"odinger Operators on the Discrete Real Line: Resolvent Boundary Values, Limiting Absorption Principle, H\"older Regularity and Dispersive Estimates      }
\begin{document}

\maketitle

\begin{abstract}

This work establishes new results on spectral theory and time evolution for matrix-valued discrete Schr\"odinger operators on the space of square-summable matrix sequences. The matrix-valued formalism is employed to streamline notation, offering a more elegant alternative to the equivalent vector-valued framework with matrix potentials.
Our main contributions are threefold: first, we derive an explicit Wronskian-based representation for the resolvent's integral kernel; second, we prove H\"older continuity for the resolvent's boundary values; and third, we establish dispersive estimates for the time evolution.
Our approach begins with the construction of Jost solutions using Volterra equations and the transmutation operator, leading to proofs of their H\"older regularity and bounds in Wiener algebra norms. From these solutions, we obtain the explicit kernel representation. This explicit characterization - a distinctive feature of the one-dimensional setting - enables the direct computation of the resolvent's boundary values.
We subsequently establish a limiting absorption principle and demonstrate the H\"older continuity of these boundary values, achieving an improvement over classical results obtained through abstract higher-dimensional methods. Finally, this detailed resolvent characterization is leveraged to prove the dispersive estimates for the time evolution of the system.


\end{abstract}

\section{Introduction}
In this work, we study boundary values of the resolvent, limiting absorption principle (together with H\"older continuity), and dispersive estimates for discrete matrix-valued Schr\"odinger operators. These operators are of the form
\begin{align}
    H = H_0 + V
\end{align}
where \( H_{0} \) is the discrete (Laplacian) Hamiltonian and \( V \) is a self-adjoint matrix multiplication operator (referred to as the potential $V$). These operators are defined in the space $\ell ^{2}$ of matrix-valued sequences that are square summable (with respect to the Frobenius inner product). See Section \ref{sec1.1} for the precise definition of these operators.

For the scalar case, the direct and inverse scattering theory has been studied by several authors (see \cite{MR332065}, \cite{MR405160}, \cite{MR454437}, \cite{MR559594}, \cite{MR499269}, \cite{MR708000}, \cite{MR730128}, and \cite{MR1896882}, \cite{MR2329923}, \cite{MR2404173}, \cite{MR2448988}). Scattering theory for vector valued Schr\"odinger operators is addressed in \cite{MR557790} and more recently in \cite{MR4959158}, \cite{MR4192212}, \cite{MR4385984}, and \cite{MR4760547}. Results on dispersive estimates and limiting absorption principle in the scalar case are addressed in \cite{MR3433284}.  There are also important contributions to the direct and inverse scattering theory for discrete operators in higher dimensions, including works that establish resolvent estimates; see, for instance, \cite{MR4018760} and \cite{MR2913620}.

It is important to clarify that, although one can obtain estimates for the Green function of the perturbed Hamiltonian in terms of the free operator (using perturbative arguments), in our work, we derive a non-perturbative formula for the boundary values of the resolvent in the continuous spectrum. Our formula is highly useful because it is expressed in terms of the Wronskian. The fact that we are working with matrices (not scalars) makes our proofs considerably more technical than in the scalar case since non-commutativity plays a prominent role. Our contribution regarding this point is the matrix-valued case ($L \geq 2$). To the best of our knowledge, explicit formulas for the integral kernel of Jacobi operators with matrix-valued coefficients have not been previously reported. For the case $L=1$, formulas for the resolvent kernel in terms of the Wronskian are well established. We believe that extending these explicit formulas to the matrix-valued case is a valuable new contribution. The derivation of these boundary values is presented in Theorem \ref{LAPmain}. Our formulas allow us to prove the limiting absorption principle in Sections \ref{sectionCP}, \ref{sectionVE}, \ref{sectionWR}, \ref{SM}, \ref{GF} and \ref{sectionLAP}. In the scalar case, this is addressed in \cite{MR3433284} and \cite{MR2529950}. In comparison to \cite{MR3433284} and \cite{MR2529950}, here we also prove H\"older continuity and H\"older-continuous error bounds for the convergence rate to the boundary values, which require several new ingredients. In the one-dimensional case, H\"older continuity uses different techniques than in higher dimensions. To the best of our knowledge, in the present paper we address for the first time H\"older continuity in this setting (one dimensional and discrete). Our results (Theorem \ref{LAPmain}) improve the results of the classic book of Yafaev \cite{MR2598115}, where higher dimensions and the continuous case are considered. 

The limiting absorption principle refers to various results concerning the existence of boundary values of the resolvent operator
\begin{align}\label{RES}
R_{H}(E)=(H-E)^{-1}, \quad E \in \mathbb{C}^{\pm},
\end{align}
on the continuous spectrum of the operator \( H \). It is clear that when $E$ approaches spectral points of $E$, the norms of the corresponding resolvent operator blow up. However, if it is applied to states featuring enough decay in configuration space, the limit does exist. We introduce a method that, to the best of our knowledge, has not been used before and yields a substantial improvement of the H\"older continuity exponent via an explicit representation of the resolvent kernel. The optimal H\"older exponent obtained with the standard method is $\rho/2$, where the decay rate of the perturbation (the potential $V$) is of order $\left(\frac{1}{|n|}\right)^{1+\rho}$ as $|n| \to \infty$. This method uses the second resolvent identity
\[
R = R_0 - R_0 V R, \qquad \text{i.e. } R = R_0(1 + V R_0)^{-1}.
\]
The invertibility of $1 + V R_0$ plays a fundamental role. First, one must ensure that $V R_0$ is compact in an appropriately weighted $\ell^2$ space (more precisely we substitute $R$ by $R^\alpha=T_{-\alpha} R T_{-\alpha}$, where $(T_{-\alpha}u)(n)=(1+|n|)^{-\alpha}$, and $R_0$ by $R_0^\alpha = T_{-\alpha}R_0T_{-\alpha}$, and we avoid weighted spaces), enforcing the restriction that the H\"older exponent cannot exceed $\rho/2$. Furthermore, the standard method only implies that the set (inside the spectrum of the free operator) where $1 + V R_0$ fails to be invertible has zero Lebesgue measure (in the $\ell^2$ space we equivalently seek for invertibility of $1+T_{2\alpha}VR_0^\alpha$). In general, a considerable amount of work must be performed in order to characterize this set. In certain specific equations this set is finite or even empty, but this has to be proved for each case.

\textbf{Our approach:}  
In contrast to the standard approach, our method yields to a substantial improvement: our H\"older exponent is essentially equal to $\rho$, instead of $\rho/2$.

We introduce a method that was not used in this context to prove the H\"older continuity: we provide explicit expressions to the Green's function in terms of the Jost solutions. This is essentially equivalent to provide explicit solutions to the time evolution of the Schr\"odinger equation. Since our approach solves the full time evolution problem (instead of merely providing estimations), we significally improve the precision of the estimates.
Moreover, in our case we do not need to characterize the spectral points where $1 + R_0 V$ is not invertible. Our explicit formulas for the Green function allow us to proceed directly, without relying on the invertibility of $1+R_0V$.

The boundary formulas we derive allow us to estimate the time evolution of the solutions to the Schr\"odinger equation, when we restrict the initial states to belong to the absolutely continuous part of the Hilbert space. These estimates are called dispersive estimates. In this work (Section \ref{dispersiveesimates}), we prove dispersive estimates in $ \ell^\infty(\mathbb{Z}, \mathcal{M}_L) $ (see \eqref{ellpspaces}) for the time evolution. In Theorem \ref{EST}, we prove that
\begin{equation}
\label{eq:decay}
    \|e^{-itH}P_{ac}\|_{\ell^{1} \to \ell^{\infty}} = O(t^{-\frac{1}{3}}), \quad t \to \infty,
\end{equation}
where $P_{ac}$ is the projection on the absolutely continuous part of the spectrum and $ \| \cdot \|_{\ell^{1} \to \ell^{\infty}} $ is the norm in the space of operators from $ \ell^1(\mathbb{Z}, \mathcal{M}_L)$ to $ \ell^\infty(\mathbb{Z}, \mathcal{M}_L)$, see \eqref{ellpspaces}. Previous results are only available for the scalar case: In \cite{MR2468536}, stronger conditions on the potential are assumed ($ |V(n)| = O((1 + |n|)^{-\beta}), \: \beta > 5$), and \cite{MR2529950} assumes that $ \sum_{n \in \mathbb{Z}} |(1 + n)^{2}V(n)| < \infty $. Similar results as ours for the scalar case are presented in \cite{MR3433284} and \cite{MR3404595}. The periodic case is addressed in \cite{MR4070305}, \cite{MR4335875} and \cite{MR4944037}.

This work presents a distinct approach to the vector-valued Schr\"odinger operator by considering its action on the space $\ell ^{2}(\mathbb{Z},\mathcal{M}_L)$ of square-summable matrix-valued sequences. Our contribution is precisely the case $L \geq 2$. Moreover, our proofs rely crucially on the explicit formula for the resolvent kernel that we derive for the first time. The presence of non-commutative terms makes the derivation significantly more involved than in the scalar case. Theorem \ref{GreenIntro} provides the new structural tool enabling the extension of scalar arguments to the matrix-valued setting. As detailed in Section \ref{sec1.1}, we equip this space with a Hilbert space structure using the Frobenius inner product for matrices. This methodological choice, which is absent from the existing literature, provides a unifying formalism that simplifies the notation and allows for a cleaner and more clear exposition of our results.

For the continuous (differential) case, there is extensive literature: the scalar case is presented in \cite{MR2096737}, \cite{MR2127577}, \cite{MR1736195}, \cite{MR2096737}, \cite{MR3535363}, \cite{MR1980087}, and vector-valued Schr\"odinger operator are studied in \cite{MR4238229}.

The main results of this paper are presented in Section \ref{MR}, while their proofs are derived in the other sections. Section \ref{sec1.1} introduces the necessary notation for the understanding of the statements of our main theorems. An important ingredient in our proofs is an explicit formula for the resolvent operator in terms of Jost Solutions, and the corresponding boundary values (Theorems \ref{GreenIntro}, \ref{Greenlimitintro}), all other results of this paper make use this formula. The proofs of Theorems \ref{GreenIntro}, \ref{Greenlimitintro} are presented in Sections \ref{EP}, \ref{sectionWR} and \ref{GF} (where Sections \ref{EP} and \ref{sectionWR} constitute important technical ingredients for the main proofs). The other two main results of this article are the limiting absorption principle (Theorem \ref{LAPmain2}), where H\"older regularity is addressed, and the dispersive estimates (Theorem \ref{teo:estimacion}). The proof of the former is the content of Section \ref{sectionLAP} and the proof of the latter is obtained in Sections \ref{PJS}, \ref{SM}, and \ref{dispersiveesimates}, where Sections \ref{PJS}, \ref{SM} are devoted to technical aspects.

\begin{obs} \label{COP}

Observe that \( \mathcal{M}_L \) can be embedded into \( B(\mathcal{M}_L) \) via multiplication operators, thereby inducing an operator norm on \( \mathcal{M}_L \). Since we are working in finite dimensions, the norms on \( \mathcal{M}_L \)-whether viewed as an operator or a matrix-are equivalent.

In this manuscript, constants are denoted by $C( \cdot, \cdot, \cdot,  \cdots ) ,$ where the parameters they depend on are explicitly indicated. For simplicity, we use this generic notation rather than introducing a distinct symbol for each constant. The value of these constants may change from line to line and can even differ on either side of an equation or inequality.

\end{obs}

\subsection{Operators and difference equations}\label{sec1.1}

For every normed vector space $(Y, \| \cdot  \|_{Y})$, we denote by $ B(Y) $ the vector space of bounded operators $ T :  Y \to Y $ equipped with the operator  norm:  
\begin{align}
    \|  T \|_{B(Y)} \equiv    \|  T \|. 
    \end{align}
We routinely omit the subscript $B(Y)$ if no confusion arises. The space of $L \times L$ complex matrices is denoted by $\mathcal{M}_L \equiv B( \mathbb{C}^L) $, and $\mathcal{M}_L^{\mathbb{Z}}$ represents the space of functions (sequences) from $\mathbb{Z}$ to $\mathcal{M}_L$. The difference expression $\tau_0: \mathcal{M}_L^{\mathbb{Z}} \to \mathcal{M}_L^{\mathbb{Z}} $ is defined by

\begin{align}\label{Edso1}
(\tau_{0}u)(n ) := u(n+1) + u(n-1).      
\end{align}
In this manuscript, we consider a sequence of self-adjoint matrices $V=(V(n))_{n\in \mathbb{Z}}\in \mathcal{M}_{L}^{\mathbb{Z}} $ such that 
\begin{align} \label{firstfinitemoment}
  \|  V \|_{0}  :=      \sum_{n}   \| V(n) \| < \infty.    
\end{align}
We define for $ \rho > 0 $  (here we regard $V(n)$
 as a multiplication operator on $ \mathcal{M}_L$)
\begin{align} \label{normV1}
{  \|  V \|_{\rho}  :=       \sum_{n}(|n| + 1) }^{\rho}\| V(n) \|,   \hspace{1cm}       \|  V \|_{\infty}  :=\sup_{n}   \| V(n) \|  .  \end{align} 
We analyze the difference operator   $\tau: \mathcal{M}_L^{\mathbb{Z}} \to  \mathcal{M}_L^{\mathbb{Z}} $, given by 
\begin{align}\label{Edso1}
(\tau u)(n ) :=   (\tau_{0}u)(n ) + V(n) u(n).    
\end{align}
Let 
\begin{align}\label{ell2}
\mathcal{H} : = \ell^{2}(\mathbb{Z}, \mathcal{M}_{L})
\end{align}
be the Hilbert space of square-summable sequences, endowed with the scalar product defined for all \(u, v \in \mathcal{H}\) by
\begin{align}
    \left<u,v\right>_{\ell ^{2}}:=\sum _{n\in \mathbb{Z}}\left<u(n), v(n)\right>_{\mathcal{M}_L},
\end{align}
where the product for \(A, B \in \mathcal{M}_L\) is the standard Frobenius product, defined as
\begin{align}
    \left<A,B\right>_{\mathcal{M}_L}=\operatorname{tr} (B^{*}A). 
\end{align}

Similarly, we denote by $\ell^{p}(\mathbb{Z}, \mathcal{M}_{L})$ the vector spaces of sequences with values in $\mathcal{M}_{L}$ that are $p$-summable for $1 \le p < \infty$, and bounded for $p = \infty$, equipped with the respective norms:
\begin{align}\label{ellpspaces}
    \|u\|_{\ell ^{p}}:=  \Big ( \sum _n \|u(n)\|_{\mathcal{M}_L}^{p} \Big )^{1/p}, \hspace{1cm}
\|u\|_{\ell ^{\infty}}:=\sup _n \{\|u(n)\|_{\mathcal{M}_L}\}.
\end{align}
\color{black}

We define the Schr\"odinger operator $H$ and the free Schr\"odinger operator $H_0$ as the restrictions of $\tau$ and $\tau_0$ to $\mathcal{H}$, respectively:
\begin{align}
    H := \tau|_{ \mathcal{H} },    \qquad
    H_0 := \tau_0|_{ \mathcal{H} }.
\end{align}
Moreover, we employ an abuse of notation and identify
\begin{align}
    V \equiv  H-H_0,
\end{align}
and we call $V$ the interaction. All these operators are bounded and self-adjoint. Indeed, one can readily verify that
 \begin{equation}
        \|H\| \leq (2+\|V\|_{\infty }),\,\,\, u\in \ell ^{2}(\mathbb{Z}, \mathcal{M}_{L}). 
    \end{equation}
Furthermore, \(V\) is a compact operator, being the norm limit of the finite-rank operators \(V \chi_{\{|n| \leq M\}}\). Consequently, the essential spectrum of \(H\) and \(H_0\) coincide:
$ 
\sigma _{\text{ess}}(H) = \sigma _{\text{ess}}(H_{0}).
$

We define the discrete Fourier transform as the operator
$ 
\mathcal{F}: \ell^{2}(\mathbb{Z}, \mathcal{M}_{L}) \longrightarrow L^{2}([-\pi, \pi], \mathcal{M}_{L})
$ 
given by
\begin{equation}\label{Fourier}
(\mathcal{F}u)(k) := \frac{1}{\sqrt{2 \pi }} \sum_{n \in \mathbb{Z}} u(n)e^{ik n}, \quad k \in  [-\pi, \pi],
\end{equation} 
for $u = \big(u(n)\big)_{n \in \mathbb{Z}}$.
 The inverse of the Fourier transform is given by
\begin{equation}
    \mathcal{F}^{-1}\phi (n)= \frac{1}{\sqrt{2 \pi }}     \int _{-\pi }^{\pi }e^{-ink}\phi (k )dk .
\end{equation}
 The discrete Fourier transform is key to studying the free Schr\"odinger operator \(H_0\) because it diagonalizes this operator. Specifically,
\begin{equation} \label{Fourier transform}
     \mathcal{F} H_{0} \mathcal{F}^{-1} \phi (k) = 2\cos(k) \phi(k),
\end{equation}
see \cite{MR4385984}, Section 1.1.1. The unitarity of the Fourier transform implies that the spectrum of \(H_0\) is purely absolutely continuous and given by \(\sigma(H_{0}) = [-2,2]\).

The spectrum of the operator \(H\),  consists of the interval \([-2, 2]\) along with finitely many eigenvalues located outside this interval, see \cite{MR4760547}. Moreover, \cite{MR4760547} proves that the point spectrum \(\sigma_p(H)\) is finite and does not intersect \([-2,2]\).

\subsection{Main results}\label{MR}

A useful parametrization of the complex plane is given by the Zhukovsky  transform:
\begin{align}\label{J}
 \mathbb{C} \ni  E : = J(z) = z+ 1/z,    \hspace{1cm}z\in \overline{\mathbb{D}}\setminus \{0\},  \: \: \text{where} \: 
 \: 
 \: \mathbb{D}:=\{z\in \mathbb{C}:|z|< 1\}. 
 \end{align}

The restriction of the Zhukovsky map \(J\) to \(\mathbb{D} \setminus \{ 0 \}\) (see \eqref{J}) is invertible onto \(\mathbb{C} \setminus [-2,2]\). We denote its inverse by \( r : \mathbb{C} \setminus [-2,2] \to \mathbb{D}\setminus \{0 \} \).
We extend this function to \(\mathbb{C}\) by defining
\begin{align} \label{r_+}
    r_+(z) = \begin{cases}
        r(z), & z \notin [-2,2], \\
        e^{-ik}, & z = 2 \cos(k), \quad k \in [0,\pi].
    \end{cases}
\end{align}
\color{black}
This extension is essential for analyzing the boundary values of the resolvent operator's matrix elements (see Section \ref{ResGreen}).

For every \(E \in \mathbb{C}\), the difference equations
\begin{align}\label{eigen}
 \tau u = E u, \qquad \tau_0 u = E u,
 \end{align}
are called generalized eigenvalue problems, and their solutions are generalized eigenvectors. When a solution \(u\) belongs to \(\mathcal{H}\), it is an eigenvector of the Hamiltonians \(H\) and \(H_0\), respectively. For \( z \in \overline{\mathbb{D}} \setminus \{0\} \), we define the function \( u^z_0 : \mathbb{Z} \to \mathcal{M}_L \) by
\begin{align}\label{JostLibre}
    u^z_0(n) = z^n I,
\end{align}
where \(I\) is the \(L\times L\) identity matrix. This function solves the generalized eigenvalue problem for the free Hamiltonian:
\begin{align}
    H_0 u_0^z = E u_0^z, \qquad \text{with } E = J(z).
\end{align}
The function \( u_0^z \) is called a free Jost solution. For the perturbed operator, Jost solutions are generalized eigenvectors that behave asymptotically like the free Jost solutions.

\begin{defi}[Jost Solutions]

\label{jost}  For every \( z \in \overline{\mathbb{D}} \setminus \{0\} \) with \( E = J(z) \), we denote by \( u_+^z \) and \( u_-^{1/z} \) the solutions to the generalized eigenvalue problem \( \tau u = E u \) that satisfy the asymptotic conditions
\begin{equation}
\label{eq:as}
    u_{+}^z(n)= z^{n}(I + o(1)), \quad \text{as } n \to +\infty, \qquad
    u_-^{1/z}(n) = (1/z)^n(I + o(1)), \quad \text{as } n \to -\infty.
\end{equation}
These solutions are called the Jost solutions.
\end{defi}
\begin{obs}
    In the literature, it is also common to define the Jost solutions as the generalized eigenvalues that satisfy the asymptotic behaviour:
    \begin{equation}
\label{eq:as}
  U_{n}^{+}(z)=z^n(I+o(1)), \ \text{as} \ n \to + \infty,
    \qquad
   U_n^{-}(z)= z^{-n} (I+o(1)), \ \text{as} \ n \to -\infty.
\end{equation}
    
    Our notation ($u_+^z, u_-^{1/z}$) is adopted primarily because the asymptotic behavior in $z$ is directly read for the notation itself: by using the variables $z$ and $1/z$, the notation itself suggests the nature of the behavior in the limits $n \to \pm \infty$, which aids the reader in remembering these asymptotic properties.
\end{obs}
\color{black}

We recall that the resolvent operator of \(H\) (see Definition \ref{PRdefi}, \eqref{RES}) is given by
\begin{align}
    R_{H}(E) = (H - E)^{-1}, \quad E \in \mathbb{C} \setminus \sigma(H).
\end{align}
Our first main result, Theorem \ref{formulaGreenmatrixelements}, provides an explicit formula for the matrix elements \(\{[R_{H}(E)]_{s,r}\}_{r,s\in \mathbb{Z}}\) (see \eqref{me}) of \(R_{H}(E)\) in terms of the Jost solutions. Its proof is presented in Section \ref{GF}.

\begin{defi}\label{wronskidefi}
For two sequences $v, w \in \mathcal{M}_L^{\mathbb{Z}}$, we define their Wronskian as the sequence given by
\begin{align}
W(v, w)(n) := i\big( v(n+1)^{} w(n) - v(n)^{} w(n+1) \big), \qquad n \in \mathbb{Z}.
\end{align}
\end{defi}
For $u, v \in \mathcal{M}_{L}^{\mathbb{Z}}$ satisfying $\tau u = Eu$ and $\tau v = \overline{E}v$, the Wronskian is independent of $n$, in this case we identify
\begin{align}
W(v,u) \equiv W(v,u)(n).
\end{align}
      From the definition of the Wronskian, for the sequences $u$ and $v$, we obtain:
\begin{align}
    W(uM, v + w)(n) = M^{*} \big( W(u, v)(n) + W(u, w)(n) \big) \\\notag
    W(u + v, w M)(n) = \big( W(u, w)(n) + W(v, w)(n) \big) M
\end{align}
and 
\begin{align}
    W(u,v)^{*}=W(v,w)
\end{align}
   
 Further details and properties are discussed in Section \ref{sectionWR}.

\begin{teo}\label{GreenIntro}
Let \(E \in \mathbb{C} \setminus \sigma(H)\) and set \(E = J(z)\). The kernel of the resolvent \(R_{H}(E)\) is given by
\begin{equation}
\label{eq:Greenintro}
[R_{H}(E)]_{s,r} =
 \begin{cases}
    -i\, u_{+}^{z}(s)\left(W(u_{-}^{1/\overline{z}}, u_{+}^{z})\right)^{-1} u_{-}^{1/\overline{z}}(r)^{*}, & s \geq r, \\[10pt]
    i\, u_{-}^{1/z}(s)\left(W(u_{+}^{\overline{z}}, u_{-}^{1/z})\right)^{-1} u_{+}^{\overline{z}}(r)^{*}, & s < r.
 \end{cases}
\end{equation}

\end{teo}

\color{black}
The continuity of the Jost solutions implies the following theorem (see Section \ref{GF} for the proof):
\begin{teo}\label{Greenlimitintro}
    For all $r,s \in \mathbb{Z}$ and $E \in (-2,2)$, the boundary values of the matrix elements $[R_{H}(\cdot)]_{r,s}$ of the resolvent operator exist and are given by:

    \begin{align}
        [R_{H}(E \pm i0)]_{s,r} &:= \lim_{\epsilon \to \pm 0} [R_{H}(E \pm i\epsilon)]_{s,r} = -iu_{+}^{z^{\pm 1}}(s)\left(W(u_{-}^{z^{\pm 1}}, u_{+}^{z^{\pm 1}})\right)^{-1} u_{-}^{z^{\pm 1}}(r)^{*}, \,\,\, s \geq r
    \end{align}

    \begin{align}
        [R_{H}(E \pm i0)]_{s,r} &:= \lim_{\epsilon \to \pm 0} [R_{H}(E \pm i\epsilon)]_{s,r} = iu_{-}^{z^{\mp 1}}(s)\left(W(u_+^{z^{\mp 1}}, u_-^{z^{\mp 1}})\right)^{-1} u_{+}^{z^{\mp 1}}(r)^{*},\,\,\, s<r
    \end{align}
    where $z = r^+(E) \in \mathbb{S}^{1}$.
\end{teo}
In the scalar case, a formula analogous to \eqref{eq:Greenintro} can be found in \cite{MR1711536}, Eq. (1.99). For our setting of matrix-valued difference equations, however, the proof is considerably more complicated due to the essential role of non-commutativity. The demonstration involves several technical ingredients, including an analysis of the properties of the transfer matrices; see Section \ref{EP}.

For every  $ E \in \mathbb{C} \setminus (\sigma_p(H)  \cup \{-2,2 \})   $, we denote by 
\begin{align}
   [ R^{\pm }_H(E) ]_{s,r}: = \begin{cases}
       [ R(E) ]_{s,r},  & \text{if   $ E \notin (-2,2) $}, \\ 
       [ R(E \pm i0)]_{r,s}  & \text{if   $ E \in (-2,2) $},    \end{cases}
\end{align}
and define the corresponding integral operator  $ R^{\pm }_H(E)$ by
\begin{align}
 (R^{\pm }_H(E) u)(s) : = &   \sum_{r}  [ R^{\pm }_H(E) ]_{s,r} u(r),   \hspace{1cm}   
  u \in D(R^{\pm }_H(E)) , \\ \notag   
\hspace{1cm}   D(R^{\pm }_H(E)) :  = & \Big \{  u \in \mathcal{M}_{L}^{\mathbb{Z}}  :   \sum_{r}  \| R^{\pm }_H(E) ]_{s,r} u(r) \|_{\mathcal{M}_L} < \infty  , \forall s \in  \mathbb{Z}  \:   \Big \}. 
\end{align}
Remark \ref{RKsr} implies that $ D(R^{\pm }_H(E)) $ includes $  \ell^2(\mathbb{Z}, \mathcal{M}_L) $ for $  E \notin (-2, 2) $, and it extends $R_H(E)$,  and Proposition \ref{LAP2}
 implies that it includes   $\ell^1(\mathbb{Z}, \mathcal{M}_L) $  for $  E \in (-2, 2) .$



Given $\alpha > 0$, we define the operator $T_{-\alpha} : \ell^{2}(\mathbb{Z}, \mathcal{M}_L) \longrightarrow \ell^{2}(\mathbb{Z}, \mathcal{M}_L)$ by
\begin{align}\label{Tsigma}
(T_{-\alpha} u)(n) = (1+|n|)^{-\alpha} u(n), \quad u = (u(n))_{n \in \mathbb{Z}} \in \ell^{2}(\mathbb{Z}, \mathcal{M}_L).
\end{align}

In Section \ref{sectionLAP} we prove that $ T_{-\alpha} R_{H}(E) T_{-\alpha} $ defines a bounded operator on $\ell^{2}(\mathbb{Z}, \mathcal{M}_L)$; see Theorem \ref{teoaco}. The following theorem states the Limiting Absorption Principle along with H\"older continuity (the proof is given in Section \ref{sectionLAP}, Theorem \ref{LAPmain}). We formulate our result in terms of $ R_H^{\pm} $, rather than only the boundary values as is customary, because this presentation is more informative and elegant.




\begin{teo}[Limiting Absorption Principle]\label{LAPmain2} 
Suppose that $ \rho > 0 $,  $\|  V \|
  _{\rho} < \infty ,$ $\rho + 1/2<\alpha $.    For every compact set $  K^{\pm}  \subset \overline{ \mathbb{C}^{\pm} } \setminus (\sigma_p(H) \cup \{ -2, 2\} )  $, the function 
\begin{align}
   K^{\pm} \ni E \to   T_{-\alpha }R^{\pm}_{H}(E)T_{-\alpha }  \in B(\ell ^{2}(\mathbb{Z}, \mathcal{M}_L))    \end{align}
is H\"older continuous  with H\"older exponent $ \min(\rho, 1)  .$

\end{teo}

  In Theorem \ref{LAPmain}, the only constraint on $\alpha$ is $\alpha > 1/2$. This is because if $\| V \|_{\rho} < \infty$ and $\alpha > 1/2$, we can always find a $\tilde{\rho} \in (0, \rho]$ such that $\tilde{\rho} + 1/2 < \alpha$, which implies $\| V \|_{ \tilde{\rho}} < \infty$. In this case, the H\"older exponent depends on $\tilde{\rho}$ rather than $\rho$.

Theorem \ref{LAPmain} improves upon the result in \cite{MR2598115}, which deals with higher dimensions and the continuous case. The fact that we can express everything explicitly in terms of Jost solutions---which is feasible in one dimension---allows us to obtain the H\"older exponent $ \min(\rho, 1) $. Assuming $\rho \leq 1$, our H\"older exponent is $\rho$, whereas the best exponent in \cite{MR2598115} is essentially $\rho/2$ (note that our notation differs from that in \cite{MR2598115}).

\color{black}

 In Chapter \ref{dispersiveesimates}, we prove the following theorem for the Schr\"odinger operator:

\begin{teo}[Dispersive Estimates] \label{teo:estimacion}
Suppose that \( \| V\|_1 < \infty \) (see \eqref{normV1}) and that \(V\) is generic in the sense of Definition \ref{generic}. Let \(P_{ac}\) denote the projection in \(\ell^{2}( \mathbb{Z}, \mathcal{M}_L )\) onto the absolutely continuous subspace of \(H\). Then, \(e^{-itH}P_{ac}\) defines a bounded operator from \( \ell^1( \mathbb{Z}, \mathcal{M}_L ) \) to \(\ell^\infty( \mathbb{Z}, \mathcal{M}_L )\), and it satisfies the estimate
\begin{equation} 
\label{eq:EST}
    \|e^{-itH}P_{ac}\|_{\ell^{1} \to \ell^{\infty}} \leq C \left( \frac{1}{1 + |t|} \right)^{\frac{1}{3}}, \quad t \geq 0,
\end{equation}
where \(C\) is a constant.
\end{teo}

\color{black}

\section{The generalized eigenvalue problem }\label{EP}
In this section we seek for solutions to the generalized eigenvalue problem \eqref{eigen},
$
 \tau u = E u.
$
 
\subsection{Cauchy problem} \label{sectionCP}

In this section, we study the Cauchy problem for the eigenvalue equation \eqref{eigen}. That is, we seek solutions to the eigenvalue problem satisfying the initial conditions
\begin{align} \label{initial}
    u(n_0) = A_0, \quad u(n_0 + 1) = B_0.
\end{align}
We follow the approach of \cite{MR4192212}, Section 2. For the reader's convenience, we recall below the relevant notation and sone results (without proofs) from that source, beginning with the following definition from the start of Section 2 in \cite{MR4192212}.
\begin{defi}\label{transmatrix}
    For each $n \in \mathbb{Z}$ and $V \in \mathcal{M}_L^{\mathbb{Z}}$, we define the transfer matrix:
\begin{align}
    \mathcal{T}^{E}(n) = \begin{pmatrix}
        E - V(n) & -I \\
        I & 0
    \end{pmatrix}.
\end{align}
For $n > m$, we define:
\begin{align}
    \mathcal{T}^{E}(n, m) := \mathcal{T}^{E}(n) \mathcal{T}^{E}(n-1) \cdots \mathcal{T}^{E}(m+1),
\end{align}
and for $n < m$:
\begin{align}\label{TI}
    \mathcal{T}^{E}(m, n) := \mathcal{T}^{E}(n, m)^{-1},
\end{align}
with the convention that $\mathcal{T}^{E}(n, n)$ is the identity matrix in $\mathcal{M}_{2L}$.
\end{defi}

These matrices yield the following propagation relations:
For any $n, m \in \mathbb{Z}$ and any solution $u$ of $\tau u =Eu $, we have:
    \begin{align}\label{transfer}
        \begin{pmatrix} u(n+1) \\ u(n) \end{pmatrix} &=
        \mathcal{T}^{E}(n) \begin{pmatrix} u(n) \\ u(n-1) \end{pmatrix}, \hspace{1cm}
        \begin{pmatrix} u(n+1) \\ u(n) \end{pmatrix} &=
        \mathcal{T}^{E}(n, m) \begin{pmatrix} u(m+1) \\ u(m) \end{pmatrix}.
    \end{align}

Defining the state vector:
\begin{align}
    \Phi(n) := \begin{pmatrix} u(n+1) \\ u(n) \end{pmatrix},
\end{align}
equation \eqref{transfer} allows us to rewrite the eigenvalue equation \eqref{eigen}, as:
\begin{align}\label{rw}
    \Phi(n) = \mathcal{T}^{E}(n) \Phi(n-1) = \mathcal{T}^{E}(n, m) \Phi(m).
\end{align}

Consider the  matrix $\mathcal{J} \in \mathcal{M}_{2L}$:
\begin{align}
    \mathcal{J} = \begin{pmatrix}
        0 & -I \\
        I & 0
    \end{pmatrix},
\end{align}
where $0$ denotes the zero matrix in $\mathcal{M}_L$. Note that:
\begin{align}
    \mathcal{J}^{-1} = \mathcal{J}^* = \begin{pmatrix}
        0 & I \\
        -I & 0
    \end{pmatrix}.
\end{align}
For any block matrix $
    \mathcal{S} = \begin{pmatrix}
        A_{11} & A_{12} \\
        A_{21} & A_{22} 
    \end{pmatrix}$ ,
we have the conjugation rule:
\begin{align}\label{irule}
    \mathcal{J} \mathcal{S} \mathcal{J}^* = \begin{pmatrix}
        A_{22} & -A_{21} \\
        -A_{12} & A_{11}
    \end{pmatrix}.
\end{align}
\begin{propo}\label{Uequivalence}
    For $E\in \mathbb{C}$ and $n\in \mathbb{Z}$ it holds that 
    \begin{align}
        \mathcal{T}^{E}(n)^{-1}=\mathcal{J}\mathcal{T}^{\overline{E}}(n)^{*}\mathcal{J}^{-1}, \hspace{1cm}
        \mathcal{T}^{E}(n,m)^{-1}=\mathcal{J}\mathcal{T}^{\overline{E}}(n,m)^{*}\mathcal{J}^{-1}.
    \end{align}
\end{propo}
 \begin{proof} We prove the result for $n > m $, for $n < m$ use \eqref{TI}.  From Definition \ref{transmatrix} and using that $V(n)$ is self-adjoint, for every $n$, we have
    $
   \mathcal{T}^{\overline{E}}(n)^{*}= \begin{pmatrix}
 E-V(n) & I\\
 -I & 0
\end{pmatrix}
$  
and using \eqref{irule} 
 \begin{align}
  \mathcal{J} \mathcal{T}^{\overline{E}}(n)^{*}\mathcal{J}^{*}= \begin{pmatrix}
 0 & I\\
 -I &  E-V(n)
\end{pmatrix}.
\end{align}  
Finally 
\begin{align}
  \mathcal{T}^{E}(n)(\mathcal{J} \mathcal{T}^{\overline{E}}(n)^{*}\mathcal{J}^{*})= 
  \begin{pmatrix}
 E-V(n) & -I\\
 I & 0
\end{pmatrix}
  \begin{pmatrix}
 0 & I\\
 -I &  E-V(n)
\end{pmatrix}=
\begin{pmatrix}
 I & 0\\
 0 &  I
\end{pmatrix}
\end{align}  

Now, for the second part, note that 
\begin{align}
     \mathcal{J} \mathcal{T}^{\overline{E}}(n,m)^{*}\mathcal{J}^{*}&=\notag
     \mathcal{J} \mathcal{T}^{\overline{E}}(m+1)^{*}\mathcal{J}^{*} \mathcal{J} \mathcal{T}^{\overline{E}}(m)^{*}\mathcal{J}^{*}\cdot \cdot \cdot \mathcal{J}  \mathcal{T}^{\overline{E}}(n)^{*}\mathcal{J}^{*}\\&\notag
     =\mathcal{T}^{E}(m+1)^{-1} \mathcal{T}^{E}(m)^{-1}\cdot \cdot \cdot \mathcal{T}^{E}(n)^{-1}\\& \notag =\mathcal{T}^{E}(n,m)^{-1}.
\end{align}
 \end{proof}

\begin{teo}
   Let \(u, w, \hat{u}, \hat{w} \in \mathcal{M}_L^{\mathbb{Z}}\) be matrix sequences such that \(u\) and \(w\) are solutions of \eqref{eigen} with parameter \(E\), while \(\hat{u}\) and \(\hat{w}\) are solutions of \eqref{eigen} with parameter \(\overline{E}\). Then, for all \(n, m \in \mathbb{Z}\), the following identity holds:
    \begin{align}\label{JPhi}
        \mathcal{J}\Phi(\hat{u}, \hat{w})(n)^{*}\mathcal{J}^{*}\Phi(u, w)(n) = \mathcal{J}\Phi(\hat{u}, \hat{w})(m)^{*}\mathcal{J}^{*}\Phi(u, w)(m)
    \end{align}
where
\begin{align}\label{Phi}
    \Phi(u, v)(n) = \begin{pmatrix}
        u(n+1) & v(n+1) \\
        u(n) & v(n)
    \end{pmatrix}.
\end{align}
\end{teo}
\begin{proof}
    Note that for $u, w \in \mathcal{M}_{L}^{\mathbb{Z}}$ solutions of \eqref{eigen}, we have from \eqref{rw}:
    \begin{align}\label{eq:uw}  
        \mathcal{T}^{E}(n, m)\Phi(u, w)(m) = \Phi(u, w)(n).
    \end{align}
    Using \eqref{eq:uw} and Proposition \ref{Uequivalence}, we obtain:
    \begin{align}
        \mathcal{J}\Phi(\hat{u}, \hat{w})(n)^{*}\mathcal{J}^{*}\Phi(u, w)(n)&= \mathcal{J}(\mathcal{T}^{\overline{E}}(n, m)\Phi(\hat{u}, \hat{w})(m))^{*}\mathcal{J}^{*}\Phi(u, w)(n) \\
        &\quad \notag= \mathcal{J}(\Phi(\hat{u}, \hat{w})(m)^{*}\mathcal{T}^{\overline{E}}(n, m)^{*})\mathcal{J}^{*}\Phi(u, w)(n) \\
        &\quad \notag= \mathcal{J}\Phi(\hat{u}, \hat{w})(m)^{*}\mathcal{J}^{*}\mathcal{J}\mathcal{T}^{\overline{E}}(n, m)^{*}\mathcal{J}^{*}\Phi(u, w)(n) \\
        &\quad \notag= \mathcal{J}\Phi(\hat{u}, \hat{w})(m)^{*}\mathcal{J}^{*}\mathcal{T}^{E}(n, m)^{-1}\Phi(u, w)(n) \\
        &\quad \notag= \mathcal{J}\Phi(\hat{u}, \hat{w})(m)^{*}\mathcal{J}^{*}\Phi(u, w)(m) .
    \end{align}
\end{proof}
 \begin{teo}\label{cauhyproblem}
   Let \(A_{0}, B_{0} \in \mathcal{M}_{L}\) and \(n_{0} \in \mathbb{Z}\). There exists a unique solution \(u\) of the eigenvalue equation \eqref{eigen} satisfying \(u(n_{0}) = A_{0}\) and \(u(n_{0}+1) = B_{0}\).
 \end{teo}
\begin{proof}
   For every $n\in \mathbb{Z}$, by equation \eqref{rw} we have:
 \begin{align}
    \begin{pmatrix}
        u(n+1)\\
        u(n)
    \end{pmatrix}=\mathcal{T}^{E}(n,n_0)
     \begin{pmatrix}
        B_0\\
        A_0
    \end{pmatrix}
\end{align}
determines the unique solution  $u(n)$ that satisfies   \eqref{eigen}.
\end{proof}

\color{black}

\color{cyan}

\color{black}

\subsection{Volterra equation and existence of Jost solutions}\label{sectionVE}

 The following theorem is established in \cite{MR4385984} (Theorem 33) and for the scalar case in \cite{MR1711536} (Lemma 7.8):
\begin{teo}\label{teo:volt}
We consider the Volterra equation
\begin{equation}
\label{eq:volt}
f(n) = g(n) + \sum_{m=n+1}^{\infty} \mathcal{K}(n,m)f(m),
\end{equation}
where \(g \in \ell^{\infty}(\mathbb{Z}^{+}, \mathcal{M}_{L})\) and \(\mathcal{K}(n,m) \in \mathcal{M}_L\) for all \(n, m \in \mathbb{Z}^{+}\). Suppose there exists a sequence \(M \in \ell^{1}(\mathbb{Z}^{+}, \mathbb{R})\) such that \(\|\mathcal{K}(n,m)\| \leq M(m)\) for all \(n, m \in \mathbb{Z}^{+}\). Then, Equation (\ref{eq:volt}) admits a unique solution \(f \in \ell^{\infty}(\mathbb{Z}^{+}, \mathcal{M}_{L})\) ($\mathbb{Z}^{+}$ denoted the set of positive integers). 
\end{teo}

As a consequence of Theorem \ref{teo:volt}, we obtain the following result from \cite{MR4385984} (Lemma 7):

\begin{lema}\label{jostvolterra}
    If \(\|V\|_1 < \infty\), then the Jost solutions defined in Definition \ref{jost} exist for all \(z \in \overline{\mathbb{D}}\setminus \{0\}\). Moreover, for each fixed \(n \in \mathbb{Z}\), the maps \(z \mapsto u_{\pm}^{z^{\pm 1}}(n)\) are continuous on \(\overline{\mathbb{D}}\setminus \{0\}\), holomorphic on \(\mathbb{D}\setminus \{0\}\), and satisfy the Volterra equations:

    \begin{align}\label{jostsolvolterra}
        u_+^z(n) &= z^n I - \sum_{j=n+1}^\infty S^z(j-n)V(j)u_+^z(j), \quad n \in \mathbb{Z},\\ \notag
        u_-^{1/z}(n) &= z^{-n} I + \sum_{j=-\infty}^{n-1} S^{1/z}(j-n)V(j)u_-^{1/z}(j), \quad n \in \mathbb{Z}
    \end{align}
    where
    \begin{equation}
\label{eq:scalar}
S^{z}(n) = \left\{ \begin{array}{lcc}
            \dfrac{z^{n}-z^{-n}}{z-z^{-1}} = \dfrac{z}{z+1}\sum_{j=-n}^{n-1} z^{j}, & z^{2} \neq 1 \\
            (\pm 1)^{n+1}n, & z = \pm 1.
            \end{array}
   \right.
   \end{equation}
\end{lema}

\begin{lema}\label{voltinK}
We assume only that $ \|  V\|_0 < \infty $. Let \(K \subseteq \overline{\mathbb{D}}\setminus \{0, -1, 1\}\) be compact . Then for all \(z \in K\), the functions \(u_{\pm}^{z^{\pm 1}}\) satisfy the Volterra equations \eqref{jostsolvolterra}, respectively, and they are continuous on \(K\).
\end{lema}

\begin{proof}
    The proof follows the same lines as that of Lemma \ref{jostvolterra}, using Lemma \ref{teo:volt} with the following key observation (we consider only the case of the plus sign).  The function $  \mathcal{K} $ in \cite{MR4385984} is 
    $-z^{m-n} S^{z}(m-n)  V(m)   $, see Lemma 36 in \cite{MR4385984}.  Away from $ z \in \{ 0, -1, 1 \}  $, it is uniformly bounded by a constant $C(K)$ times $ \|  V(m)\|$, which is summable.     
\end{proof}
\subsection{Transmutation operator and representation of Jost solutions}
\color{black}
While the existence of Jost solutions is established in \cite{MR4385984} (Lemma 7) via Volterra equations, a different construction is needed to analyze their properties in relation to the Wiener algebras introduced in Section \ref{PJS}. This alternative approach, based on the well-known transmutation (or transformation) operator, is a crucial tool for our proofs. The scalar case is treated in \cite{MR3433284}, and the operator-valued version in \cite{sher2025scatteringtheorydifferenceequations}.

For the reader's convenience, we provide a detailed introduction to the necessary notation and state (without proofs) results from   \cite{sher2025scatteringtheorydifferenceequations}. The primary goal of this section is to derive uniform bounds for the Jost solutions in terms of Wiener algebra norms.

\begin{defi} \label{B+}
Assume that $\|  V\|_1  < \infty.$ For every $n\in \mathbb{Z}$ and $m\in \mathbb{N}$ we define the next recursive relations:
\begin{align}
B_0^+(n) = I ,  \hspace{.3cm}B^+_1(n)&=-\sum_{\ell =n+1}^\infty  V(\ell),  \hspace{0.3cm} 
    B_2^+(n) =-\sum_{\ell =n+1}^\infty  V(\ell ) B_1^+(\ell ),  \\
    B_{m+2}^+(n)&=-\sum_{\ell =n+1}^\infty V(\ell )B_{m+1}^+(\ell ) +B_m^+(n+1).
\end{align}
The convergence of the above sums is a consequence of $ \| V \|_1 < \infty$. Actually, it is straightforward to inductively  prove that, for every  fixed  $ m   $, $   \| B^{+}_{m}(n)\|  $ is uniformly bounded with respect to $n $.     Additionally, we define the following recursive relations:
\begin{align}
  B_0^{-}(n) = I, \hspace{.3cm}  B^-_1(n)&=-\sum_{\ell =-\infty}^{n-1} V(\ell ), \hspace{.3cm}
    B_2^-(n) =-\sum_{\ell =-\infty}^{n-1}  V(\ell ) B_1^-(\ell ), \\
    B_{m+2}^-(n)&=-\sum_{\ell=-\infty}^{n-1} V(\ell )B_{m+1}^+(\ell ) +B_m^-(n+1).
\end{align}
    
\end{defi}
Furthermore, it is proven inductively (\cite{sher2025scatteringtheorydifferenceequations}, Corollary 3.5) that for every $n \in \mathbb{Z}$, the following estimate holds (recall Remark \ref{COP} ):
\begin{align}\label{SumaBs}
    \lVert B_m^+(n)\rVert_{\mathcal{M}_L} \leq C_n \sum_{\ell=n+\lfloor\frac{m}{2}\rfloor}^\infty \lVert V(\ell ) \rVert ,  \hspace{.5cm}
     C_n:=C^5\exp\big(C^3\sum_{\ell =n+1}^\infty (\ell-n)\lVert V(\ell )\rVert \big) ,    \end{align}
where $C$ is a constant large enough such that $\sum_{n\in \mathbb{Z}}\lVert V(n)\rVert <C$. 

\begin{lema}\label{LemmaSumVs}   Assume that $\|  V\|_1  < \infty.$  The following inequality holds true (recall the norms \eqref{normV1}):
\begin{align}
\sum _{j=0}^{\infty}\sum _{\ell =n+j}^{\infty}\|V(\ell )\| \leq 2 \| V  \|_{1}  -  \min(n, 0) \| V \|_0. 
\end{align}

\end{lema}

\begin{proof}


For a positive sequence one has the identity 
\begin{equation}
    \sum_{j=0}^\infty \sum_{l=n+j}^\infty a_{j,l}= \sum_{l=n}^\infty \sum_{j=0}^{l-n} a_{j,l}.
\end{equation}
Applying this identity one obtains
\begin{equation}
\sum _{j=0}^{\infty}\sum _{\ell =n+j}^{\infty}\|V(\ell )\| = \sum_{l=n}^\infty  \sum_{j=0}^{l-n} \|V(l)\| = \sum_{l=n}^\infty \|V(l)\| (l-n). 
\end{equation}
From this equation one obtains the result.

\end{proof}

\begin{propo}
\label{ub} Assume that $\|  V\|_1  < \infty.$ There are positive constants $ \mathbf{b}_{\pm} $ and  $  \mathbf{c}_{\pm} $ such that for all $  n \in \mathbb{Z}$
\begin{align}
\sum_{m = 1 }^{\infty} \|   B^{\pm}_{m}(n)  \|_{\mathcal{M}_L
} \leq \mathbf{b}_{\pm} (1 +  e^{ \mp \mathbf{c}_{\pm}  n}). 
\end{align}

\end{propo}

\begin{proof}
    We fist notice that for non negative  $ n $, we have that $\ell -n  \leq \ell $ and, therefore (see \eqref{normV1} and \eqref{SumaBs}),  
\begin{align}\label{BoundCnmas}
C_n \leq C^5 \exp(  C^3 \| V \|_{1}    ) . 
\end{align}
For negative $n $ (recall \eqref{normV1}), we have that  
\begin{align}\label{BoundCnmenos}
C_n \leq  & C^5 \exp(  C^3 \sum_{\ell = n+1}^\infty  |\ell| \| V(\ell) \|     ) \exp(  C^3 \sum_{\ell = n+1}^\infty  -n \| V(\ell) \|     ) 
\\ \notag 
\leq & C^5 \exp(  C^3 \| V \|_{1}    )  \exp( -n C^3 \| V \|_{0}    ). 
\end{align}



Using \eqref{SumaBs}, we estimate

\begin{equation}\label{abssumm}
    \sum _{m=1}^{\infty }\|B^{+}_m(n)\|_{\mathcal{M}_L}<C_{n}\sum _{m=1}^{\infty} \sum_{\ell=n+\lfloor\frac{m}{2}\rfloor}^\infty \lVert V(\ell ) \rVert ,
\end{equation} for all $n\in \mathbb{Z}.$

Notice that 
\begin{align}\label{abssumm1}
\sum_{m = 0}^{\infty} \sum_{\ell=n+\lfloor\frac{m}{2}\rfloor}^\infty \lVert V(\ell ) \rVert  =  2 \sum_{s = 0}^{\infty} \sum_{\ell=n+  s}^\infty \lVert V(\ell ) \rVert .\end{align}

Eqs. \eqref{abssumm} and \eqref{abssumm1} together with Lemma \ref{LemmaSumVs}  and Eqs.  \eqref{BoundCnmas}, \eqref{BoundCnmenos}  imply 
\begin{align}
 \sum _{m=1}^{\infty }\|B^{+}_m(n)\|_{\mathcal{M}_L} \leq (2 \| V  \|_{ 1} -  \min(n, 0) \| V \|_0)    C^5 \exp(  C^3 \| V \|_{1}    )  \exp( -  \min(n,0)  C^3 \| V \|_{0}    ). 
\end{align}

This implies the result for the sum of $  \| B^+_{m}(n) \|_{\mathcal{M}_L} $, since, for negative  $n$, 
$ -n $ can be bounded by a constant $C_{\epsilon}$ times $ e^{-\epsilon n} $.  
The corresponding result for the sum of  $  \| B^-_{m}(n) \|_{\mathcal{M}_L} $ is proved analogously.

\end{proof}

\begin{obs}\label{remSumBs}
Proposition \ref{ub} implies that for all  $N \in \mathbb{N}$:
\begin{align}\label{AUbound}
    \sup_{n \geq N}   \sum _{m=1}^{\infty }\|B^{+}_m(n)\|_{\mathcal{M}_L}  \leq  \mathbf{b}_{+} (1 +  e^{ - \mathbf{c}_{+}  N}),  \hspace{1cm}   \sup_{n \leq N}   \sum _{m=1}^{\infty }\|B^{-}_m(n)\|_{\mathcal{M}_L} \leq  \mathbf{b}_{-} (1 +  e^{  \mathbf{c}_{-}  N}) .   \end{align}

\end{obs}

\color{black}

Remark \ref{remSumBs} imply that, for every $z\in \overline{\mathbb{D}}$, the mappings
\begin{align}
    z \to I+\sum _{m=1}^{\infty }{B_{m}^{\pm  }(n)z^{m}}
\end{align}
define  holomorphic functions on $\mathbb{D}$ and continuous on $\overline{\mathbb{D}}$. \color{black}
\begin{teo}[Jost Solutions Representation, Therem 3.1, \cite{sher2025scatteringtheorydifferenceequations}]\label{JSR} 
Assume that $\|  V\|_1  < \infty.$ For every $z\in \overline{\mathbb{D}}  \setminus \{ 0 \}$, the follwing formula holds true:
    \begin{equation}
        \label{eq:taylor+}
        u_{\pm }^{z^{\pm 1}}(n)=z^{\pm n}\left(I+\sum _{m=1}^{\infty }{B_{m}^{\pm  }(n)z^{m}}\right).  
    \end{equation}
\label{teo:jsrepresentation}
\end{teo}

\subsection{Divergent solutions}

The  Jost solutions do not generate all solutions of the generalized eigenvalue problem \eqref{eigen}, for $|z |< 1$.  In order to generate a basis of solutions, we define other solutions that diverge asymptotically (see Definition \ref{vsolutions}). In Proposition \ref{base} below we prove this fact.

\begin{defi}\label{vsolutions}
     For every $z\in \overline{\mathbb{D}}\setminus \{-1, 0, 1\},$ we denote by  $v^{z}_{+}$ y $v^{z^{-1}}_{-}$ the solutions of the generalized eigenvalue problem (\ref{eigen}) 
satisfying 
\begin{equation}
\label{v+-}
 v^{z^{\pm 1}}_\mp(n)=z^{\pm n}(I+o(1)),\,\,\, n\to \mp \infty .
\end{equation}

\end{defi}
The existence of the solutions in \eqref{v+-} is established in Proposition 6 of \cite{MR4760547}. Notice that   if $z \in \mathbb{S}^1\setminus \{-1,  1\}$, then $ v^{z^{\pm 1}}_\mp =u^{z^{\pm 1}}_\mp$.

\begin{propo}\label{base}
For $z\in \overline{\mathbb{D}}\setminus \{-1, 0, 1\},$ and $E=J(z)$, any solution of \eqref{eigen} can be uniquely written as 
\begin{align}\label{yh}
s^{z} = u^{z}_+ A_+ + v^{z^{- 1}}_+ B_+, \hspace{1cm} 
s^{z} = u^{z^{ - 1}}_- A_- + v^{z}_- B_-
\end{align} 
where $A_\pm,B_\pm \in \mathcal{M}_{L}$. 
\end{propo}

\begin{proof}
We consider the $+$ case in \eqref{yh} without loss of generality. The function $s^z$ satisfies \eqref{eigen} because $u^{z^{\pm 1}}_\pm$ and  $v^{z^{\pm 1}}_\mp$ do. If $ |z| =1 $, the result presented in \eqref{upluminus} and the lines above it. Then we assume that $|z| < 1$. The uniqueness of $A_{+}$ and $B_+$ follows by taking the limits:
\begin{align}
    B_+ = \lim_{n \to \infty} z^n s^{z}(n), \hspace{1cm} A_+ = \lim_{n \to \infty} z^{-n}(s^z - v^{z^{-1}}_+ B_+).
\end{align}
Theorem \ref{cauhyproblem} implies that the dimension of the solution space is $2L^2$. Combined with the uniqueness of $A_+$ and $B_+$, this ensures that all solutions are of the form $s^z$.
\end{proof}

 \color{black}
 
 \section{Properties of Jost solutions}\label{PJS}

 \subsection{Wiener algebras}

We define the following set of functions: 
 \begin{equation}
 \label{eq:WM}
     \mathcal{A}_{\mathcal{M}_{L}}=\left\{ f: [-\pi, \pi ]\to \mathbb{C}: f(k)=\sum_{m\in\mathbb{Z}}a_me^{imk} \hspace{0.2cm} \text{and} \hspace{0.2cm} (a_m)_{m\in\mathbb{Z}}\in \ell^1(\mathbb{Z},\mathcal{M}_L)\right\} .
 \end{equation}
 $\mathcal{A}_{\mathcal{M}_{L}}$ is a normed algebra with norm given by:  
\begin{equation}\label{winernorm}
    \|f\|_{ \mathcal{A}_{\mathcal{M}_{L}}}:=\|(a_m)_{m\in  \mathbb{Z}}\|_{\ell ^{1}}
\end{equation}
and the non-commutative product 
\begin{align}
    (f g)(k) := f(k) g(k).  
\end{align}
Actually, if $ f = \sum_{m\in\mathbb{Z}}a_me^{imk} $ and $g(k)  =\sum_{m\in\mathbb{Z}}b_me^{imk} $
 then $ f g = \sum_{m\in\mathbb{Z}}c_me^{imk} $
 where the coefficients  $ c_m  $ are obtained by the convolution 
 \begin{align}
     c_m = \sum_i a_i b_{m-i},  \:  \text{and} \hspace{1cm}  \| f g \|_{ \mathcal{A}_{\mathcal{M}_{L}}} \leq \| f \|_{ \mathcal{A}_{\mathcal{M}_{L}}}  \| g \|_{ \mathcal{A}_{\mathcal{M}_{L}}}.
 \end{align}
 Since for $L = 1$, the algebra $\mathcal{A}_{\mathcal{M}_{L}}$ is known as the Wiener algebra (see \cite{MR3433284}), we extend this terminology to the case $L > 1$ as well.
In fact, a function belongs to $\mathcal{A}_{\mathcal{M}_{L}}$ if and only if all its matrix entries belong to the scalar Wiener algebra
\begin{align}
  \mathcal{A}_{\mathcal{M}_{1}} \equiv \mathcal{A},
 \end{align}
  because $ (a_m)_{m\in\mathbb{Z}}\in \ell^1(\mathbb{Z},\mathcal{M}_L)$ if and only if the matrix entries of $  (a_m)_{m\in\mathbb{Z}} $ belong to 
  $\ell^1(\mathbb{Z},\mathcal{M}_1) .$

In what follows, we make use of the well-known Wiener Lemma. For the reader's convenience, we state it here:
 
 \begin{lema}[Wiener] For every  $f\in \mathcal{A}$ such that $f(k)\neq 0$  for all $k \in \mathbb{R}$,  $\frac{1}{f}\in \mathcal{A}$.
 \end{lema} 
 
The following result generalizes Wiener's Lemma.
\begin{teo} 
\label{thm:wienerG}

For every  \(f \in \mathcal{A}_{\mathcal{M}_{L}}\)  such that \(\det(f(k)) \neq 0\) for all \(k \in \mathbb{R}\),  the multiplicative inverse function, \(f^{-1}\), is an element of \(\mathcal{A}_{\mathcal{M}_{L}}\).

\end{teo}

\begin{proof}
Note that  $ 
f^{-1}(k) = \frac{1}{\det(f(k))} \operatorname{adj}(f(k)), \quad \text{for all } k \in [-\pi, \pi].
$  
Since the entries of \(\operatorname{adj}(f(k))\) are obtained through linear combinations and products of the entries of \(f(k)\), it follows that these entries belong to \(\mathcal{A}\). The same holds for \(\det(f(k))\). By Wiener's Theorem, \(\frac{1}{\det(f(k))}\) also belongs to \(\mathcal{A}\), and hence all entries of \(f^{-1}\) do as well. This implies that \(f^{-1} \in \mathcal{A}_{\mathcal{M}_{L}}\).

\end{proof}

The estimates in Theorem \ref{VanderWiener} below rely on the Van der Corput Lemma. For the reader's convenience, we state it here:
\begin{teo}[Van der Corput's Lemma]\label{Vander} 
Suppose that $\phi$ is real-valued and $l$-times differentiable on $(a,b)$, with $l \geq 2$. If $|\phi^{(l)}(x)| \geq \delta > 0$ for all $x \in (a,b)$, then
$ 
\left|\int_{a}^{b} e^{it \phi(x)}  dx \right| \leq c (t \delta)^{-\frac{1}{l}}
$ 
for some constant $c$ independent of $\phi$ and $t$.
\end{teo}

\begin{teo}\label{VanderWiener}

We suppose that $ \phi : [-\pi, \pi ] \to \mathbb{R} $  is $l-$times differentiable 
($ s \geq 2$)
    and let \(f \in \mathcal{A}_{M_{L}}\), and \(-\pi \leq a < b \leq \pi\).  Assume that the $l-$derivative $\phi^{(l)}$  of $  \phi $ satisfies \(|\phi^{(l)}(k)| > \delta > 0\) for all \(k \in (a, b)\). We set $ 
I(t) = \int_{a}^{b} e^{it\phi(k)} f(k) \, dk .
$ 
It follows that there is a constant $C$, independent of $a$ and $b$,  such that  
$ 
\|I(t)\|_{\mathcal{M}_L} \leq \frac{C}{(\delta t)^{\frac{1}{l}}} \|f\|_{\mathcal{\mathcal{A}_{M_{L}}}}
$ 
where \(C\) does not depend of the interval \([a, b]\).
\label{teo:vanderG}
\end{teo}

\begin{proof}

For $f\in \mathcal{A}_{\mathcal{M}_L}$ and for $r\in \mathbb{Z}$, we have that 
$
\left\|\sum_{n=m}^{r} f(n) e^{i n k}\right\|_{\mathcal{M}_L} \leq \sum_{n=m}^{r} \|f(n)\|_{\mathcal{M}_L} .
$
Furthermore, 
\[
I(t) = \int_{a}^{b} e^{it\phi(k)} f(k) \, dk = \int_{a}^{b} e^{it\phi(k)} \sum_{n \in \mathbb{Z}} f(n) e^{i nk} \, dk = \sum_{n \in \mathbb{Z}} f(n) I_{\frac{n}{t}}(t),
\]
where 
$
I_{\frac{n}{t}}(t) = \int_{a}^{b} e^{i t \left(\phi(k) + \frac{n}{t} k\right)} \, dk.
$
Applying Van der Corput's Lemma (Lemma \ref{Vander}) to each \(I_{\frac{n}{t}}(t)\), the result follows.
\end{proof}

\begin{defi}\label{utilde} For $z\in \overline{\mathbb{D}}\setminus \{0\}$ we define 
\begin{align}
    \tilde{u}_{\pm }^{z^{\pm 1}}(n):=z^{\mp n}u_{\pm }^{z^{\pm 1}}(n),\,\,\, n\in \mathbb{Z}.
\end{align} 
    
\end{defi}

\begin{propo} \label{Abound}

For every $ n \in \mathbb{Z} $, the functions $  \mathbb{S}^{1}\ni z \mapsto u^{z^{\pm 1}}_{\pm}(n)$
and  $  \mathbb{S}^{1}\ni z \mapsto \tilde{u}^{z^{\pm 1}}_{\pm}(n)$ (see Definition \ref{jost}) belong to the Wiener Algebra $   \mathcal{A}_{\mathcal{M}_L}  $ (see \eqref{eq:WM}).  Moreover, 
\begin{align}
\|u^{z^{\pm 1}}_{\pm}(n)\|_{ \mathcal{A}_{\mathcal{M}_L} } \leq L^2+ L\mathbf{b}_{\pm }(1+e^{\mp \mathbf{c_\pm }n}),
\hspace{0.5cm}
\|\tilde{u}^{z^{\pm 1}}_{\pm}(n)\|_{ \mathcal{A}_{\mathcal{M}_L} } \leq L+   \mathbf{b}_{\pm }(1+e^{\mp \mathbf{c_\pm }n}),
\end{align}
where $\mathbf{b}_\pm$, $\mathbf{c}_\pm$, are the constants introduced in Lemma \ref{ub}.

\end{propo}

\begin{proof}
We use Theorem \ref{teo:jsrepresentation} to obtain 
    $ 
        \tilde{u}_{\pm }^{z^{\pm 1}}(n)=I+\sum _{m=1}^{\infty }B_{m}^{\pm  }(n)z^{m}.  
    $ 
This is the expression of $ \tilde u_{\pm }^{z^{\pm 1}}(n) $ as a Fourier transform of the form $ \tilde u_{\pm }^{z^{\pm 1}}(n) = \sum _{r} {a_{r} z^{r}}$, with $a_{0} = I$, $ a_{r} = 0 $ for negative $r$ and $ a_r = B^{\pm}_r(n) $, for $r \geq 1$. Then we have that 
\begin{align}
    \left\|\tilde{u}^{z^{\pm 1}}_{\pm}(n)\right\|_{ \mathcal{A}_{\mathcal{M}_L} } = \sum_r \|a_r\|_{\mathcal{M}_L} \leq L+ \mathbf{b}_{\pm} (1 +  e^{ \mp \mathbf{c}_{\pm}  n}),    
    \end{align}
where we use Lemma \ref{ub}.   Since $ u_{\pm }^{z^{\pm 1}}(n) =  z^{\pm n} \tilde{u}_{\pm }^{z^{\pm 1}}(n)$, we have that 
 $ \| u_{\pm }^{z^{\pm 1}}(n) \|_{ \mathcal{A}_{\mathcal{M}_L} }\leq \|  z^{\pm n} \|_{ \mathcal{A}_{\mathcal{M}_L} }\| \tilde{u}_{\pm }^{z^{\pm 1}}(n)  \|_{ \mathcal{A}_{\mathcal{M}_L} }$
and the result holds true for $u_{\pm }^{z^{\pm 1}}(n).$
\end{proof}
\begin{cor}
    Given $ N \in \mathbb{Z} $, there are  positive constants $\mathbf{A}_N ^{\pm}$ such that
    \begin{align}\label{one}
\left\|\tilde{u}^{z}_{+}(n)\right\|_{ \mathcal{A}_{\mathcal{M}_L} } \leq \mathbf{A}_N ^{+}, \,\,\,
\text{for}\,\,\, n\geq N, \: \, \: \left\|\tilde{u}^{z^{-1}}_{-}(n)\right\|_{ \mathcal{A}_{\mathcal{M}_L} } \leq \mathbf{A}_N ^{-},\,\,\, \text{for}\,\,\, n\leq N . \end{align}
 for $z\in \mathbb{S}^{1}.$
    \end{cor}
    
\begin{proof}
    We take $ n  $ and $ N $ as in the statement of the present Corollary.  From Proposition \ref{Abound} and \eqref{AUbound} we have that 
    $
\left\|\tilde{u}^{z^{\pm 1}}_{\pm}(n)\right\|_{ \mathcal{A}_{\mathcal{M}_L} } \leq L+\mathbf{b}_{\pm }(1+e^{\mp \mathbf{c_\pm }n}) \leq  L+ \mathbf{b}_{\pm} (1 +  e^{ - \mathbf{c}_{\pm }  N}). 
$
 Taking $\textbf{A}^{\pm}_{N}:=L+ \mathbf{b}_{\pm} (1 +  e^{ - \mathbf{c}_{\pm }  N})$ the result is follows as desired.
\end{proof}
\color{black}
\section{The Wronskian and some technical results  }\label{sectionWR}

Recall the definition of the Wronskian (see Def. \ref{wronskidefi}):

\begin{defi}
    For $u,v\in \mathcal{M}_{L}^{\mathbb{Z}}$, the Wronskian is defined by 
    \begin{equation}
        W(u,v)(n):=i\big(u(n+1)^{*}v(n)-u(n)^{*}v(n+1)\big),\,\,\,n\in \mathbb{Z}
    \end{equation}
    \end{defi}
    \begin{obs}\label{ND}
It is straightforward to verify that 
\begin{align} \label{symw}
W(u,v)(n)^{*} = W(v,u)(n).
\end{align}
For $u, v \in \mathcal{M}_{L}^{\mathbb{Z}}$ satisfying $\tau u = Eu$ and $\tau v = \overline{E}v$, the Wronskian is independent of $n$ (see Lemma 12 in \cite{MR4192212}). Therefore, in this case we identify
\begin{align}
W(v,u) \equiv W(v,u)(n).
\end{align}
\end{obs}
   \begin{obs}
      From the definition of the Wronskian, we obtain:
\begin{align}\label{WProp}
    W(uM, v + w)(n) = M^{*} \big( W(u, v)(n) + W(u, w)(n) \big) \\\notag
    W(u + v, w M)(n) = \big( W(u, w)(n) + W(v, w)(n) \big) M
\end{align}
   \end{obs}
\begin{teo}
    Let $u, w, \hat{u}, \hat{w} \in \mathcal{M}_L^{\mathbb{Z}}$ such that $u, w$ are solutions of $\tau u = Eu $ and $\hat{u}, \hat{w}$ are solutions of $\tau u = \overline{E}u $. Then:
    \begin{align}\label{PhiW}
        \mathcal{J}\Phi(\hat{u}, \hat{w})(n)^{*}\mathcal{J}^{*}\Phi(u, w)(n) = \frac{1}{i}
        \begin{pmatrix}
            -W(\hat{w}, u)(n) & -W(\hat{w}, w)(n) \\
            W(\hat{u}, u)(n) & W(\hat{u}, w)(n)
        \end{pmatrix}
    \end{align}
\end{teo}

\begin{proof}
    Using \eqref{irule} and \eqref{Phi}  we obtain:
    \begin{align}
      \mathcal{J}\Phi (\hat{u}, \hat{w})(m)^*\mathcal{J}^{*}=
      \begin{pmatrix}
          \hat{w}(m)^{*}&-\hat{w}(m+1)^{*}\\
          -\hat{u}(m)^{*}&\hat{u}(m+1)^{*}
      \end{pmatrix}
    \end{align} and this equation implies
    \begin{align}
        \mathcal{J}\Phi (\hat{u}, \hat{w})(m)\mathcal{J}^{*}\Phi (u,w)(m)=
       \begin{pmatrix}
          \hat{w}(m)^{*}&-\hat{w}(m +1 )^{*}\\ \notag
          -\hat{u}(m)^{*}&\hat{u}(m + 1)^{*}
      \end{pmatrix}
        \begin{pmatrix}
            u(m+1)&w(m+1)\\
            u(m)&w(m)
        \end{pmatrix}\\ 
        =\begin{pmatrix}
           \hat{w}(m)^{*} u(m+1)-\hat{w}(m+1)^{*}u(m)&\hat{w}(m)^{*}w(m+1)-\hat{w}(m+1)^{*}w(m)\\
           -\hat{u}(m)^{*}u(m+1)+\hat{u}(m+1)^{*}u(m)&
           -\hat{u}(m)^{*}w(m+1)+\hat{u}(m+1)^{*}w(m)
        \end{pmatrix}\\
        =\frac{1}{i}
             \begin{pmatrix}
            -W(\hat{w}, u)(m) & -W(\hat{w}, w)(m) \\
            W(\hat{u}, u)(m) & W(\hat{u}, w)(m)
        \end{pmatrix}.
       \end{align}
       
\end{proof}
 \begin{propo}\label{y5}
For $z \in \mathbb{S}^{1}$, we have:
 \begin{equation}
 \label{nu}
W(u_{\pm}^{\overline{z}^{\pm 1}}, u_{\pm }^{z^{\pm 1}})=0, \hspace{1cm} W(u_{\pm}^{z}, u_{\pm }^{z})=(\nu ^{z})^{-1}I     \,\,\,  \text{(for $z \notin \{ 1, -1 \} $)}\:   , \: \: \: 
 \end{equation}
 where
 \begin{equation}\label{nu2}
     \nu ^{z}=\frac{i}{z-z^{-1}}.
 \end{equation}

\end{propo}

\begin{proof}
We prove the first equality, the second is analogous. From Definition \ref{wronskidefi} and the asymptotic behavior of the Jost solutions as $n \to \infty$, we obtain:
\begin{align}\label{y4}
W(u_{+}^{\overline{z}}, u_{+}^{z}) = i\left(z^{n+1}(I + o(1))z^{n}(I + o(1)) - z^{n}(I + o(1))z^{n+1}(I + o(1))\right).
\end{align}
Taking the limit as $n \to \infty$ in \eqref{y4} yields the desired result.

The other cases follow similarly.
\end{proof}
 \begin{teo} \label{Winvertible}
For $J(z) = E \in \mathbb{C} \setminus ( \sigma_p(H) \cup\{ -2, 2 \}) $, the Wronskians
$ 
W(u_-^{\overline{z}^{-1}}, u_+^{z})$   and  $W(u_+^{\overline{z}}, u_-^{z^{-1}})$ 
are invertible.
\end{teo}

\begin{proof}
We show that $W(u_+^{\overline{z}}, u_-^{z^{-1}})$ is invertible  for $|z|<1$, the case $|z|=1$ is presented in \eqref{MN} and the lines below it; the proof for $W(u_-^{\overline{z}^{-1}}, u_+^{z})$ is analogous. By Proposition \ref{base}, there exist matrices $A$ and $B$ such that 
\begin{align}\label{aabb}
u_-^{z^{-1}} = u_+^{z}A + v_+^{z^{-1}}B.
\end{align}

Now observe that (see \eqref{eq:as} and \eqref{v+-})
\begin{align}
W(u_+^{\overline{z}}, v_+^{z}) = i\left(z^{n+1}(I+o(1))z^{-n}(I+o(1)) - z^{n}(I+o(1))z^{-n-1}(I+o(1))\right).
\end{align}
Taking the limit as $n \to \infty$ yields
\begin{align}\label{1234}
W(u_+^{\overline{z}}, v_+^{z}) = i(z - z^{-1}).
\end{align}

Using the properties of the Wronskian \eqref{WProp}, Proposition \ref{y5}, \eqref{aabb} and \eqref{1234}, we compute:
\begin{align}
W(u_+^{\overline{z}}, u_-^{z^{-1}}) = W(u_+^{\overline{z}}, u_+^{z}A + v_+^{z}B) = W(u_+^{\overline{z}}, u_+^{z})A + W(u_+^{\overline{z}}, v_+^{z})B 
= i(z - z^{-1})B.
\end{align}

Suppose $\overline{x} \in \ker W(u_+^{\overline{z}}, u_-^{z^{-1}})$. Then $\overline{x} \in \ker i(z - z^{-1})B$, and hence $\overline{x} \in \ker B$. By \eqref{aabb}, we have
$ 
u_-^{z^{-1}}\overline{x} = u_+^{z}A\overline{x},
$ 
which implies that $u_-^{z^{-1}}\overline{x}$ is square-summable and satisfies $Hu_-^{z^{-1}}\overline{x} = Eu_-^{z^{-1}}\overline{x}$. Since $E \in \mathbb{C} \setminus \sigma_p(H)$, it follows that $\overline{x} = 0$; otherwise, $u_-^{z^{-1}}\overline{x}$ would be an eigenvector of $H$ with eigenvalue $E$, contradicting $E \notin \sigma _p(H)$. Therefore, $W(u_+^{\overline{z}}, u_-^{z^{-1}})$ is invertible.
\end{proof}

 \section{Scattering: transmission and reflection matrices}\label{SM}
In Section 3 of \cite{MR4385984}, it is shown that both $(u_{-}^{z}, u_{-}^{z^{-1}})$ and $(u_{+}^{z}, u_{+}^{z^{-1}})$ generate the space of solutions to $\tau u = Eu$, for $z \in \mathbb{S}^1 \setminus \{-1, 1 \}$ . Consequently, there exist matrices $M_{\pm}^{z}$ and $N_{\pm}^{z}$ satisfying (see Section 3 of \cite{MR4385984})
\begin{align}
\label{upluminus}
    u_{+}^{z}=u_{-}^{z}M_{+}^{z}+u_{-}^{z^{-1}}N_{+}^{z},\hspace{1cm}  u_{-}^{z^{-1}}=u_{+}^{z}M_{-}^{z}+u_{+}^{z^{-1}}N_{-}^{z}(u).
\end{align}
 
 Equation (\ref{nu}) leads (for $z \in \mathbb{S}^1 \setminus \{ -1, 1 \}$)  to

\begin{align}\label{MN}
    M^{z}_{+}=\nu ^{z}W(u^{z}_{-}, u^{z}_{+}) \hspace{1cm}
    N^{z}_{+}=-\nu ^{z}W(u^{z^{-1}}_{-}, u^{z}_{+})   \\
    M^{z}_{-}=-\nu ^{z}W(u^{z^{-1}}_{+}, u^{z^{-1}}_{-})\hspace{1cm}
    N^{z}_{-}=\nu ^{z}W(u^{z}_{+}, u^{z^{-1}}_{-}) . \notag
\end{align}
    
 In Proposition 16 of \cite{MR4385984}, it is shown that $M_{\pm}^{z}$ is invertible for $z \in \mathbb{S}^1 \setminus \{-1, 1\}$. In this case, we can rewrite equations (\ref{upluminus}) as
\begin{equation}
\label{TReq}
      u^{z}_{+}T^{z}_{+}= u^{z}_{- }-u^{z^{-1}}_{- }R^{z}_{+},\,\,\,
      u^{z^{-1}}_{-}T^{z}_{-}=u^{z^{-1}}_{+}-u^{z}_{+}R^{z}_{-}
  \end{equation}
where
\begin{equation}
\label{eq:TR1}
    T^{z}_{\pm}:=(M^{z}_{\pm})^{-1}, \,\,\, R^{z}_{\pm}:=-N^{z}_{\pm}(M^{z}_{\pm})^{-1}.
\end{equation}

  \begin{obs}
     From \eqref{MN} and \eqref{symw}, it follows that:
      \begin{equation}
      \label{Tproperty}
          (T^{\overline{z}}_{+})^{*}=T^{z}_{-}.
      \end{equation}
  \end{obs}

\section{Green functions and the resolvent operator}\label{GF}

Every  operator $\mathcal{K}\in B(\ell^{2}(\mathbb{Z},\mathcal{M}_{L}) )$ has a representation of the form
\begin{equation}
(\mathcal{K}u)(s) = \sum_{r\in \mathbb{Z}} [\mathcal{K}]_{s,r}u(r),
\end{equation}
where $[\mathcal{K}]_{s,r} : \mathcal{M}_{L} \to \mathcal{M}_{L}$ is uniformly bounded (with respect to $s, r$) and $u \in \ell^{2}(\mathbb{Z},\mathcal{M}_{L})$. To establish this representation, we introduce the following operators from \cite{MR4192212}, equation (13).

\begin{defi}
For each $s \in \mathbb{Z}$, we define the operator $\pi_{s}: \ell^{2}(\mathbb{Z},\mathcal{M}_{L}) \longrightarrow \mathcal{M}_{L}$ by
\begin{equation}
\pi_{s}(u) = u(s),
\end{equation}
for every $u \in \ell^{2}(\mathbb{Z},\mathcal{M}_{L})$.
\end{defi}

\begin{obs}
For each $s \in \mathbb{Z}$, the operator $\pi_{s}$ is bounded, and its adjoint $\pi_{s}^{*}: \mathcal{M}_{L} \longrightarrow \ell^{2}(\mathbb{Z},\mathcal{M}_{L})$ is given by
\begin{equation}
(\pi_{s}^{*}(A))(r) = \delta_{s,r}A,
\end{equation}
for $A \in \mathcal{M}_{L}$, where $\delta_{s,r}$ denotes the Kronecker delta.
\end{obs}

\begin{defi}\label{DKsr}
For every bounded operator $\mathcal{K}: \ell^{2}(\mathbb{Z},\mathcal{M}_{L}) \longrightarrow \ell^{2}(\mathbb{Z},\mathcal{M}_{L})$, we define
\[
[\mathcal{K}]_{s,r} = \pi_{s}\mathcal{K}\pi_{r}^{*} \in B(\mathcal{M}_L).
\]
\end{defi}

\begin{obs}\label{RKsr}
For every $s$ and $r$, the sequences $([\mathcal{K}]_{s,x})_{x \in \mathbb{Z}}$ and $([\mathcal{K}]_{x,r})_{x \in \mathbb{Z}}$ are square-summable. Indeed, for any $A \in \mathcal{M}_L$, define $u_{r,A} := \pi^{*}_{r}(A) \in \ell^{2}(\mathbb{Z},\mathcal{M}_{L})$. Then
\begin{align}\label{KER}
\|([\mathcal{K}]_{x,r}A)_{x \in \mathbb{Z}}\|_{\ell^{2}} = \|\mathcal{K}u_{r,A}\|_{\ell^{2}} \leq \|\mathcal{K}\| \|A\|_{\mathcal{M}_L}.
\end{align}
This implies that $([\mathcal{K}]_{x,r})_{x \in \mathbb{Z}}$ is square-summable. The square-summability of $([\mathcal{K}]_{s,x})_{x \in \mathbb{Z}}$ follows by noting that
$ 
[\mathcal{K}]_{s,x} = \left([\mathcal{K}^*]_{x,s}\right)^*.
$ 

Consequently, for every $u \in \ell^{2}(\mathbb{Z},\mathcal{M}_{L})$, we have $\sum_{r\in \mathbb{Z}} \|[\mathcal{K}]_{s,r}u(r)\|_{\mathcal{M}_L} < \infty$ and
\begin{align}
(\mathcal{K}u)(s) = \sum_{r\in \mathbb{Z}} [\mathcal{K}]_{s,r}u(r), \quad u \in \ell^{2}(\mathbb{Z},\mathcal{M}_{L}).
\end{align}
\end{obs}

\color{black}



\color{black}

\subsection{Parametrization  of the Riemann  Surface $ J^{-1} $.}\label{ResGreen}



    For every $E\in \mathbb{C} \setminus \{-2, 2 \}$, there exists two unique  zeros  $z_E   $ and $z_E^{-1}$ to the function  
    $$ \mathbb{C} \setminus \{ 0 \} \ni z \to  E - J(z) =: h_E(z)  = E - z- 1/z  . $$ 
    Solving for $z$ reveals that these are given by a two-valued function, corresponding to the two-sheeted Riemann surface of the square root:
      \begin{align}\label{Inv}
          E  \to J^{-1}(E) := E/2 + \frac{1}{2} \sqrt{ E^2 -4   }.  
      \end{align}
The symbol \( J^{-1}(E) \) is not a function but a multi-valued function representing a two-sheeted Riemann surface, similar to the square root; here we use an abuse of notation for the sake of clarity.       We denote by 
\begin{align}
   \mathbb{S}^{1}_{\pm}=\{z\in \mathbb{S}^{1}: \pm \Im (z)\geq 0\}. 
\end{align}
It follows that \( J \) maps \( \mathbb{S}^{1}_{\pm} \) injectively onto \( [-2, 2] \), since \( J(e^{ik}) = 2 \cos(k) \). Equation \eqref{Inv} implies that \( |J^{-1}(E)| = 1 \) for all \( E \in [-2,2] \), regardless of the specific branch of \( J^{-1}(E) \). Thus, we conclude that \( J(E) \in [-2,2] \iff E \in \mathbb{S}^1 \).  

Moreover, since \eqref{Inv} is a multi-valued inverse of \( J \), it follows that \( J \) maps \( \mathbb{C} \setminus \{0\} \) onto \( \mathbb{C} \). Given that \( J(z) = J(1/z) \), \( J \) injectively maps \( \mathbb{D}\setminus \{0\} \) onto \( \mathbb{C} \setminus [-2, 2] \), and also injectively maps \( \mathbb{C} \setminus \overline{\mathbb{D}} \) onto \( \mathbb{C} \setminus [-2, 2] \). We denote by  
$ 
r : \mathbb{C} \setminus [-2,2] \to \mathbb{D}\setminus \{0\}
$ 
the inverse of \( J|_{\mathbb{D} \setminus \{0\}} \). As the inverse of an analytic function is analytic, \( r \) is analytic.   Both \( r(E) \) and \( r(E)^{-1} \) satisfy the equation  
$ 
E = r(E) + r(E)^{-1} = J(E).
$  
Based on the above discussion, \( r \) can be extended to a bijective function on \( \mathbb{C} \) in the following two ways: we define  
\begin{align} \label{r_-}
    r_-(z) = \begin{cases}
        r(z),   \: \: \text{if}\,\,\,  z \notin [-2,2] ,
        \\ e^{ik},   \: \: \text{if} \,\,\,  z = 2 \cos(k) , k \in [0,\pi].  
    \end{cases}
\end{align}

\begin{align} \label{r_+}
    r_+(z) = \begin{cases}
        r(z),   \: \: \text{if}\,\,\,  z \notin [-2,2] ,
        \\ e^{-ik},   \: \: \text{if}\,\,\, z   = 2 \cos(k) , k \in [0,\pi].  
    \end{cases}
\end{align}
\begin{lema} \label{rpm}
The functions $r_\pm $ defined by \eqref{r_+} and \eqref{r_-} satisfy the following: 
\begin{align}
    r_{\pm}(a) = \lim_{z \to a, \pm \Im z > 0 } r(z), \hspace{1cm} \forall a \in (-2,2).   
\end{align}

\end{lema}

\begin{proof}
The branch points of the Riemann surface of \( J \) are \( -2 \) and \( 2 \), as they correspond to the branch points of the square root in \eqref{Inv}. On every simply connected region excluding these branch points, one can define two analytic inverse functions of \( J \) (since the Riemann surface is two-sheeted). In particular, there exist two analytic inverse functions of \( J \) defined on an open set containing the upper complex plane and the interval \( (-2, 2) \). One of these must coincide with \( r \) in the upper complex plane. Therefore, the limit  
\begin{align}\label{limm}  
\lim_{\zeta \to E, \, \Im \zeta > 0} r(\zeta) := a  
\end{align}  
exists. Moreover, for every \( \zeta \in \mathbb{C} \) with \( \Im \zeta > 0 \), we have \( \Im r(\zeta) < 0 \). This follows from the identity (recalling that \( J(r(\zeta)) = \zeta \))  
\[  
0 < \Im \zeta = \Im \left( r(\zeta) + r(\zeta)^{-1} \right) = \left( 1 - \frac{1}{|r(\zeta)|^2} \right) \Im r(\zeta)  
\]  
and the fact that \( 0 < |r(\zeta)| < 1 \). Since the limit in \eqref{limm} must satisfy \( J(a) = E \), there exists \( k \in (0, \pi) \) such that \( a = e^{ik} \) or \( a = e^{-ik} \). The above argument implies \( \Im(a) < 0 \) (using \( E \notin \{-2, 2\} \)), so \( a = e^{-ik} \). The statement for \( r_- \) is proved similarly.

\end{proof}

\subsection{Resolvent operator and Green functions boundary values}\label{secres}
\begin{defi}\label{PRdefi}
For every $E \in \rho(H) := \mathbb{C} \setminus \sigma(H)$, we define the  resolvent by
\begin{align}\label{perturbedresolvent}
R_{H}(E) := (H - E)^{-1}.
\end{align}
The matrix elements of this operator are given by
\begin{align}\label{me}
[R_{H}(E)]_{s, r} := \pi_s R_{H}(E)\pi_r^{*}.
\end{align}
\end{defi}

By Definition \ref{DKsr} and Remark \ref{RKsr}, it follows that $\{[R_{H}(E)]_{s, r}\}_{r, s \in \mathbb{Z}}$ is the integral kernel of $R_{H}(E)$, i.e.,
\begin{equation}\label{integraloperator}
(R_{H}(E) u)(s) = \sum_{r \in \mathbb{Z}} [R_{H}(E)]_{s,r} u(r), \quad E \in \rho(H).
\end{equation}
Moreover, we observe that $(H - E) R_{H}(E) \pi_r^{*} = \pi_r^{*}$, and therefore the following difference equation holds:
\begin{equation}
\label{eq:kernel}
[R_{H}(E)]_{s+1,r} + [R_{H}(E)]_{s-1,r} + (V(s) - E) [R_{H}(E)]_{s,r} = \pi_s \pi_r^{*}.
\end{equation}

We now proceed to derive an explicit expression for the resolvent kernel. First, we prove the following lemma:
\begin{lema}\label{aaaa}
    For every $E \in \mathbb{C}\setminus \sigma(H)$, where $E = J(z)$, the matrix
    \begin{align}\label{aaa}
        \Phi(u_+^z, u_-^{z^{-1}})(r) = 
        \begin{pmatrix}
            u_+^z(r+1) & u_-^{z^{-1}}(r+1) \\
            u_+^z(r) & u_-^{z^{-1}}(r)
        \end{pmatrix}
    \end{align}
    is invertible, and its inverse is given by:
    \begin{align}
        \Phi(u_+^z, u_-^{z^{-1}})(r)^{-1} = 
       i \begin{pmatrix}
            -W(u_-^{\overline{z}^{-1}}, u_+^z)^{-1} u_-^{\overline{z}^{-1}}(r)^* 
            & W(u_-^{\overline{z}^{-1}}, u_+^z)^{-1} u_-^{\overline{z}^{-1}}(r+1)^* \\
            -W(u_+^{\overline{z}}, u_-^{z^{-1}})^{-1} u_+^{\overline{z}}(r)^* 
            & W(u_+^{\overline{z}}, u_-^{z^{-1}})^{-1} u_+^{\overline{z}}(r+1)^*
        \end{pmatrix}.
    \end{align}
\end{lema}

\begin{proof}
    Using \eqref{PhiW} with $u_+^z$ and $u_-^{z^{-1}}$ (solutions of $\tau u = Eu $) and $u_+^{\overline{z}}$ and $u_-^{\overline{z}^{-1}}$ (solutions of $\tau u = \overline{E}u $), we obtain:
    \begin{align}\label{wronskian_identity}
        \mathcal{J}\Phi(u_+^{\overline{z}}, u_-^{\overline{z}^{-1}})(r)^*\mathcal{J}^*\Phi(u_+^z, u_-^{z^{-1}})(r) 
        &\notag = \frac{1}{i}
        \begin{pmatrix}
            -W(u_-^{\overline{z}^{-1}}, u_+^z)(r) & -W(u_-^{\overline{z}^{-1}}, u_-^{z^{-1}})(r) \\
            W(u_+^{\overline{z}}, u_+^z)(r) & W(u_+^{\overline{z}}, u_-^{z^{-1}})(r)
        \end{pmatrix} \\
        & = \frac{1}{i}
        \begin{pmatrix}
            -W(u_-^{\overline{z}^{-1}}, u_+^z) & 0 \\
            0 & W(u_+^{\overline{z}}, u_-^{z^{-1}})
        \end{pmatrix},
    \end{align}
    where the off-diagonal terms vanish due to Proposition \ref{y5}, which implies $W(u_-^{\overline{z}^{-1}}, u_-^{z^{-1}}) = 0$ and $W(u_+^{\overline{z}}, u_+^z) = 0$. Since the diagonal entries of the matrix on the right-hand side of \eqref{wronskian_identity} are invertible for Theorem \ref{Winvertible}, it follows that:
    \begin{align}
        \Phi(u_+^z, u_-^{z^{-1}})(r)^{-1} &\notag = i
        \begin{pmatrix}
            -W(u_-^{\overline{z}^{-1}}, u_+^z)^{-1} & 0 \\
            0 & W(u_+^{\overline{z}}, u_-^{z^{-1}})^{-1}
        \end{pmatrix}
        \mathcal{J}\Phi(u_+^{\overline{z}}, u_-^{\overline{z}^{-1}})(r)^*\mathcal{J}^* \\
        & = i
        \begin{pmatrix}
            -W(u_-^{\overline{z}^{-1}}, u_+^z)^{-1} u_-^{\overline{z}^{-1}}(r)^* 
            & W(u_-^{\overline{z}^{-1}}, u_+^z)^{-1} u_-^{\overline{z}^{-1}}(r +1 )^* \\
            -W(u_+^{\overline{z}}, u_-^{z^{-1}})^{-1} u_+^{\overline{z}}(r)^* 
            & W(u_+^{\overline{z}}, u_-^{z^{-1}})^{-1} u_+^{\overline{z}}(r+1 )^*
        \end{pmatrix}.
    \end{align}
    The last equality follows from a direct computation using the definition of $\Phi(u, v)$ and the identity \eqref{irule}.
\end{proof}

In the results that follow, the matrix multiplication operator $L_A \in B(\mathcal{M}_L)$, defined by $L_A(B) = AB$ for all $B \in \mathcal{M}_L$, will be denoted simply by $A$ when no confusion arises.
\color{black}
\begin{teo}\label{formulaGreenmatrixelements}
    For  $E\in \mathbb{C}\setminus \sigma (H)$, the kernel of $R_{H}$ is given by
    \begin{equation}
    [R_{H}(E)]_{s,r}=
    \label{eq:Green}
 \left\{ \begin{array}{lcc}
               -iu_{+}^{z}(s) W(u_{-}^{\overline{z}^{-1}},u_{+}^{z},)^{-1} u_{-}^{\overline{z}^{-1}}(r)^{*} &   s\geq r \\
            iu_{-}^{z^{-1}}(s) W(u_{+}^{\overline{z}}, u_{-}^{z^{-1}})^{-1} u_{+}^{\overline{z}}(r)^{*} &  s< r
             \end{array}
   \right.
   \end{equation}
   for $J(z)=E$ where $z\in \mathbb{D}\setminus \{0\}$. 
   
   The matrix elements, $[R_{H}(E)]_{s,r}$, are called Green functions.
\end{teo}
\begin{proof}
   Regarding equation \eqref{eq:kernel}, the following holds for all $C\in \mathcal{M}_L$:
\begin{equation}
     \label{eq:kernel2}
     [R_{H}(E)]_{ s+1,r}C + [R_{H}(E)]_{ s-1,r}C + (V(s) - E)[R_{H}(E)]_{s,r}C = \delta_{r,s}  C.
\end{equation}
In particular, for \(s = r\), we obtain:
\begin{equation}\label{deq}
     [R_{H}(E)]_{r+1,r}C + [R_{H}(E)]_{r-1,r}C + (V(r) - E)[R_{H}(E)]_{r,r}C = C.
\end{equation}

On the other hand, for a fixed \(r \in \mathbb{Z}\), the sequence \( s \mapsto [R_{H}(E)]_{s,r}C\), for $s>r$, satisfies the equation
\begin{align}
[R_{H}(E)]_{ s+1,r}C + [R_{H}(E)]_{ s-1,r}C + (V(s) - E)[R_{H}(E)]_{s,r}C=0 . 
\end{align}
The solution to the Cauchy problem 
\begin{align}
    \tau w = E w ,   \hspace{1cm}  w(r+1) = [R_{H}(E)]_{r+1,r}C,   \hspace{.3cm}
     w(r+2) = [R_{H}(E)]_{r+2,r}C
     \end{align}
is unique and it is given in Theorem \ref{cauhyproblem}. Then, $w$ is a unique extension of 
   \( \Big (  [R_{H}(E)]_{s,r}C \Big )_ {s>r}\),  to $ \mathbb{Z} $. 
 Recall that by Proposition \ref{base}, there exist matrices \(A^{z}, B^{z} \in \mathcal{M}_{L}\) such that:
\begin{align}
    w(s) = u_{+}^{z}(s) A^{z} + v_{+}^{z^{-1}}(s) D^{z}.
\end{align}
Since the sequence \( s \mapsto [R_{H}(E)]_{s,r}C\) is square summable (see Remark \ref{RKsr})  and \(v_{+}^{z^{-1}}(s) = z^{-s}(I + o(1))\) as \(n \to \infty\), we have that:
 \begin{align}\label{nj}
  w(s) =u_{+}^{z}(s) A^{z}, \:  \: \forall s,  \hspace{.5cm}   [R_{H}(E)]_{s,r}C =u_{+}^{z}(s) A^{z},\,\,\, s\geq r+1  
 \end{align}
($A^z$ and $D^z$ depend on $r$, but for simplicity in the proof, we will omit it).

Substituting \(s = r+1\) in \eqref{eq:kernel2}, we obtain:
\begin{align}\label{g}
     [R_{H}(E)]_{r+2,r} + [R_{H}(E)]_{r,r} + (V(r+1) - E)[R_{H}(E)]_{ r+1,r} = 0
\end{align}
  and this equation is equivalent to 
  \begin{align}\label{g1}
      [R_{H}(E)]_{r,r} =-  [R_{H}(E)]_{r+2,r}-(V(r+1) - E)[R_{H}(E)]_{r+1,r}.
\end{align}
Now we use the fact that $w$ satisfies the equation  \eqref{nj}, that  $\tau w=Ew$, that $  w $ is an extension   of  \( \Big (  [R_{H}(E)]_{s,r}C \Big )_ {s>r}\) and  \eqref{g1} to obtain 
    \begin{align}\label{nj2}
u_{+}^{z}(r) A^{z} = w(r) = - w(r+2) -(V(r+1) - E ) w(r+1) =      [R_{H}(E)]_{r,r}C .
\end{align}
From equations \eqref{nj} and \eqref{nj2}, we conclude:
\begin{align} \label{nj3}
     [R_{H}(E)]_{s,r}C = u_{+}^{z}(s) A^{z}, \quad \text{for } s \geq r.
\end{align}
A similar analysis yields there there is a matrix $B^z$ such that 
\begin{align} \label{nj4}
     [R_{H}(E)]_{s,r}C= u_{-}^{z^{-1}}(s) B^{z}, \quad \text{for } s \leq r.
\end{align}

Now, substituting \eqref{nj3} and \eqref{nj4} into \eqref{deq} for \(s = r\), we get:
\begin{align}\label{nj5}
     u_{+}^{z}(r+1) A^{z} + u_{-}^{z^{-1}}(r-1) B^{z} + (V(r) - E) u_{-}^{z^{-1}}(r) B^{z} = C.
\end{align}
Since \(u_{-}^{z^{-1}}\) satisfies \(\tau u = Eu \), equation \eqref{nj5} simplifies to:
\begin{align}\label{s1}
     u_{+}^{z}(r+1) A^{z} - u_{-}^{z^{-1}}(r+1) B^{z} = C.
\end{align}
Additionally, from \eqref{nj3} and \eqref{nj4} evaluated at \(s = r\), we have:
\begin{align}\label{s2}
     u_{+}^{z}(r) A^{z} = u_{-}^{z^{-1}}(r) B^{z}.
\end{align}
Finally, from \eqref{s1} and \eqref{s2}, we obtain the following system of equations:
\begin{align}
     u_{+}^{z}(r+1) A^{z} - u_{-}^{z^{-1}}(r+1) B^{z} &= C, \\
     u_{+}^{z}(r) A^{z} - u_{-}^{z^{-1}}(r) B^{z} &= 0 .
\end{align}
This can be writen in the next form: 
\begin{align}
    \Phi ( u_{+}^{z}, u_{-}^{z^{-1}})(r)
    \begin{pmatrix}
        A^{z}\\
        -B^{z}
         \end{pmatrix}=
          \begin{pmatrix}
        C\\
        0
         \end{pmatrix},
\end{align}
where we use \eqref{aaa}.  Now we recall  Lemma \ref{aaaa} to obtain
\begin{align}
    \begin{pmatrix}
        A^{z}\\
       - B^{z}
         \end{pmatrix}= i\begin{pmatrix}
           -W(u_- ^{\overline{z}^{-1}}, u_+^{z})^{-1}u_- ^{\overline{z}^{-1}}(r)^{*} & W(u_- ^{\overline{z}^{-1}}, u_+^{z})^{-1}u_- ^{\overline{z}^{-1}}(r+1)^{*} \\
           -W(u_+ ^{\overline{z}}, u_-^{z^{-1}})^{-1}u_+ ^{\overline{z}}(r)^{*} & W(u_+ ^{\overline{z}}, u_-^{z^{-1}})^{-1}u_+ ^{\overline{z}}(r+1)^{*}
        \end{pmatrix} \begin{pmatrix}
        C\\
        0
         \end{pmatrix}.
\end{align}

Finally, we use  \eqref{nj3} and \eqref{nj4} to get
 \begin{equation}
    [R_{H}(E)]_{s,r}C=
 \left\{ \begin{array}{lcc}
             -iu_{+}^{z}(s) W(u_{-}^{\overline{z}^{-1}},u_{+}^{z},)^{-1} u_{-}^{\overline{z}^{-1}}(r)^{*}C &   s\geq r \\
            iu_{-}^{z^{-1}}(s) W(u_{+}^{\overline{z}}, u_{-}^{z^{-1}})^{-1} u_{+}^{\overline{z}}(r)^{*}C &  s< r
             \end{array}
   \right.
   \end{equation}  for all $C\in \mathcal{M}_L.$
\end{proof}
\begin{obs}
    From \eqref{eq:Green} ,  \eqref{symw},  we verify that  for all $s,r\in \mathbb{Z}$:
    \begin{align}
        [R_{H}(E)]_{r,s}=[R_{H}(\overline{E})]_{s,r}^{*},
    \end{align}
    for $E=J(z)$, where $z \in \mathbb{D}\setminus \{0, 1, -1\}$, see also below \eqref{KER}. 
\end{obs}

\begin{teo}\label{Green2.1}
    For all $r,s\in \mathbb{Z}$ and for $E\in (-2,2)$, the function $ [R_{H}(E)]_{r,s}$ has boundary limits  and they are given by
\begin{align}\label{eq:Green1.1}
        [R_{H}(E\pm i0)]_{s,r}=-iu_{+}^{z^{\pm 1}}(s)(W(u_{-}^{z^{\pm 1}},  u_{+}^{z^{\pm 1}} ))^{-1} u_{-}^{z^{\pm 1}}(r)^{*}, \,\,\, \text{if}\,\,\, s\geq r,
    \end{align} 
 \begin{align}\label{eq:Green2.1}
        [R_{H}(E\pm  i0)]_{s,r}=iu_{-}^{z^{\mp 1}}(s)W(u_+^{z^{\mp 1}}, u_-^{z^{\mp 1}})^{-1}u_{+}^{z^{\mp 1}}(r)^{*},\,\,\, \text{if}\,\,\, s< r,
    \end{align}
 where $z=r^+(E)\in \mathbb{S}^{1}.$
\end{teo}
\begin{proof}
Fist we notice that the Wronskians \(W(u_{+}^{z}, u_{-}^{1/\overline{z}})\) and \(W(u_{-}^{1/z},  u_{+}^{\overline{z}})\) are invertible whenever $J(z)  \in (2, 2)$, this is consequence of \eqref{MN} and the lines below it. For the points $z$ such that
$ J(z) $ does not belong to  $  [-2, 2] $, these Wronskians are invertible  on the points $ z $ such that $J(z)$ is not a spectral point (an eigenvalue) of  $H$ (see Theorem \ref{Winvertible}). We take $ E \in (-2, 2) $  and denote by $ z_\epsilon = r(E + i \epsilon)   $,  for $ \epsilon > 0$,  see the text above \eqref{r_-} and $z = r_+(E)$. 
We use Lemma \ref{rpm},  \eqref{eq:Green}, the continuity of Jost solutions (Lemma \ref{jostvolterra}) to obtain for $ s \geq r $ (the case $ s< r $ is similar, notice also that $ \overline{z} = z^{-1} $):
    \begin{align}
 \label{eq:greenlimit}
    [R_{H}(E+ i0)]_{s,r}&=\lim _{\varepsilon \to 0^{+}}[R_{H}(E+i\varepsilon )]_{s,r}
    \\& \notag =  \lim_{\epsilon \to 0^{+} }-iu_{+}^{z_{\epsilon}}(s)(W(u_{-}^{\overline{z_
    {\epsilon}}^{-1}},  u_{+}^{z_{\epsilon}} ))^{-1} u_{-}^{\overline{ z_{\epsilon}}^{-1}}(r)^{*}    \\& \notag = -iu_{+}^{z}(s)(W(u_{-}^{z},  u_{+}^{z} ))^{-1} u_{-}^{z}(r)^{*} 
  \end{align}
  and
  \begin{align}
 \label{eq:greenlimit2}
    [R_{H}(E- i0)]_{s,r}&=\lim _{\varepsilon \to 0^{+}}[R_{H}(E-i\varepsilon )]_{s,r}\\&\notag= \lim _{\varepsilon \to 0^{+}}
     -iu_{+}^{z_{\varepsilon }}(s)(W(u_{-}^{\overline{z_{\varepsilon }}^{-1}},u_{+}^{z_{\varepsilon }}))^{-1} u_{-}^{\overline{z_{\varepsilon }}^{-1}}(r)^{*}\\&\notag= 
    -iu_{+}^{z^{-1}}(s)(W(u_{-}^{z^{-1}},u_{+}^{z^{-1}}))^{-1} u_{+}^{z^{-1}}(r)^{*} 
  \end{align}
 where $z=r_+ (E)=r_-(E)^{-1}$ are the functions given by \eqref{r_+} and \eqref{r_-}. 
\end{proof}

\begin{obs} Using  Definition \ref{utilde}  and \eqref{eq:TR1} together with Theorem \ref{Green2.1}, we prove that: 
     \begin{align}\label{eq:Green1}
        [R_{H}(E\pm i0)]_{s,r}=\frac{z^{\pm (s-r)}}{z^{\pm 1}-z^{\mp 1}}\tilde{u}_{+}^{z^{\pm 1}}(s)T_{+}^{z^{\pm 1}}\tilde{u}_{-}^{z^{\pm 1}}(r)^{*}, \,\,\, \text{if}\,\,\, s\geq r,
    \end{align}  
 \begin{align}\label{eq:Green2}
        [R_{H}(E\pm  i0)]_{s,r}=\frac{z^{\pm (r-s)}}{z^{\pm 1}-z^{\mp 1}}\tilde{u}_{-}^{z^{\mp 1}}(s)T_{-}^{z^{\pm 1}}\tilde{u}_{+}^{z^{\mp 1}}(r)^{*},\,\,\, \text{if}\,\,\, s<r . 
    \end{align}

\end{obs}

\begin{obs} For $E\in  (-2,2)$ and  $s,r\in \mathbb{Z}$, by the properties of $T^{z}_\pm $, \eqref{Tproperty}, we have

\begin{equation}
    [R_{H}(E + i0 )]_{s,r}^{*}=[R_{H}(E - i0 )]_{r,s}.
    \label{eq:simetria}
\end{equation}
\end{obs}

 \section{Limiting absorption principle}\label{sectionLAP}


 \begin{lema}\label{LAP1}
  On any compact set \(K \subseteq \overline{\mathbb{D}} \setminus \{0, -1,  1\}\), we have
    \begin{align}\label{i1}
        \sup _{s \in \mathbb{Z}, \, z \in K} \left\{ \| z^{\mp s} u_{\pm }^{z^{\pm 1}}(s) \| \right\} \leq C(K, \| V \|_0),
    \end{align} for some constant  $C(K, \| V \|_0)$. 
\end{lema}

\begin{proof} We  carry out the proof for only for the plus sign, the other case is analogous.
    In   Lemma \ref{voltinK}, it is proven that the sequence \(u_+^{z}\) satisfies the Volterra equation
    \begin{equation}
        u_{+}^{z}(s) z^{-s} = I - \sum_{r=s+1}^{\infty} z^{r-s} S^{z}(r-s) V(r) z^{-r} u_{+}^{z}(r),
    \end{equation}
  where $S^z$ is the solution given by \eqref{eq:scalar},
which implies
    \begin{equation}
        \| u_{+}^{z}(s) z^{-s} \|\leq \| I \| + \sum_{r=s+1}^{\infty} \| z^{r-s} S^{z}(r-s) V(r) \| \| z^{-r} u_{+}^{z}(r) \|.
    \end{equation}
    Applying Gronwall's Lemma (\cite{MR4760547}, Lemma 25) yields
    \begin{equation}\label{ex}
        \| u_{+}^{z}(s) z^{-s} \| \leq \| I \| \exp\left( \sum_{r=s+1}^{\infty}\| z^{r-s} S^{z}(r-s) V(r) \| \right).
    \end{equation}
Now, for \(r > s\), if \(K \subseteq \overline{\mathbb{D}} \setminus \{0, -1, 1\}\) is compact, $K\ni z\mapsto z^{r-s} S^{z}(s-r)$ is bounded (see \eqref{eq:scalar}), then, there exists a constant \(C(K)\) such that
    \begin{align}\label{ex2}
        \sup_{z \in K} \left\{ \| z^{r-s} S^{z}(r-s) V(r) \|_{\mathcal{M}_L} \right\} \leq C(K) \| V(r) \| .
    \end{align}
From \eqref{ex} and \eqref{ex2}, it follows that for  any \(z \in K\),
    \begin{align}
        \| u_+^{z}(s) z^{-s} \| \leq \| I \|  \exp\left( C(K) \| V \|_0 \right),
    \end{align}
    where $\|\hspace{.3cm}\|_0$ is the norm defined in \eqref{firstfinitemoment}.
\end{proof}
\begin{obs}\label{IQ}
    Let $z_{1}, z_{2}\in \overline{\mathbb{D}}.$ If $n$ is a non-negative integer and $0<\rho $, then 
    \begin{align}
        |z_{1}^{n}-z_{2}^{n}|   = |z_{1}^{n}-z_{2}^{n}|^{1-\min(\rho, 1)} \Big | (z_1 - z_2) \sum_{i = 0}^{n-1} z_1^i  z_2^{n-i-1} \Big |^{\min(\rho,1)} 
        \leq 2|z_1 -z_2|^{\min(\rho, 1) }|n|^{\min(\rho, 1) }.
    \end{align}
\end{obs}
\begin{lema}\label{y10}
  Suppose that $\rho > 0 $ and $ \|  V\|_{\rho} < \infty $.  
  Let $K \subset \overline{ \mathbb{D}} \setminus  \{0, -1,  1\} $ be a compact set. 
  Then for  $z \in K$, the following estimate holds true:
    \begin{align}\label{z}
        \sup_{s \in \mathbb{Z}} \left\| (1+|s|)^{-\rho } \left( z^{\mp s} u_{\pm}^{z^{\pm 1}}(s) - z_{0}^{\mp s} u_{\pm}^{z_{0}^{\pm 1}}(s) \right) \right\| &\leq C(K, \rho, \|  V\| _{\min(\rho, 1 )})  |z - z_0|^{\min(\rho, 1)}.
        \end{align}
\end{lema}

\begin{proof}
    From \eqref{jostsolvolterra}, we obtain for $z\in \overline{\mathbb{D}}\setminus \{0, 1, -1\}$:
    \begin{align}\label{b}
        &(1+|s|)^{-\min(\rho, 1 )} \left( z^{-s} u_{+}^{z}(s) - z_{0}^{-s} u_{+}^{z_{0}}(s) \right) \\
        &\notag = - (1+|s|)^{-\min(\rho, 1 )} \sum_{r=s+1}^{\infty} \left[ z^{r-s} S^{z}(r-s) V(r) z^{-r} u_{+}^{z}(r) - z_{0}^{r-s} S^{z_{0}}(r-s) V(r) z_{0}^{-r} u_{+}^{z_{0}}(r) \right] \\
        &\notag = - (1+|s|)^{-\min(\rho, 1 )} \sum_{r=s+1}^{\infty} \left( z^{r-s} S^{z}(r-s) V(r) - z_{0}^{r-s} S^{z_{0}}(r-s) V(r) \right) z^{-r} u_{+}^{z}(r) \\
        &\notag \quad - (1+|s|)^{-\min(\rho, 1 )} \sum_{r=s+1}^{\infty} (1+|r|)^{\min(\rho, 1 )} z_{0}^{r-s} S^{z_{0}}(r-s) V(r) \left( \frac{z^{-r} u_{+}^{z}(r) - z_{0}^{-r} u_{+}^{z_{0}}(r)}{(1+|r|)^{\min(\rho, 1 )}} \right).
    \end{align}
    
    We define:
    \begin{align}\label{f}
        f^{z}(s) &:= - (1+|s|)^{-\min(\rho, 1 )} \sum_{r=s+1}^{\infty} \left( z^{r-s} S^{z}(r-s) V(r) - z_{0}^{r-s} S^{z_{0}}(r-s) V(r) \right) z^{-r} u_{+}^{z}(r), \\
        g^{z}(s) &:= - (1+|s|)^{-\min(\rho, 1 )} \sum_{r=s+1}^{\infty} (1+|r|)^{\min(\rho, 1 )} z_{0}^{r-s} S^{z_{0}}(r-s) V(r) \left( \frac{z^{-r} u_{+}^{z}(r) - z_{0}^{-r} u_{+}^{z_{0}}(r)}{(1+|r|)^{\min(\rho, 1 )}} \right),
    \end{align}
    so that
    \begin{align}\label{c}
        (1+|s|)^{-\min(\rho, 1 )} \left( z^{-s} u_{+}^{z}(s) - z_{0}^{-s} u_{+}^{z_{0}}(s) \right) = f^{z}(s) + g^{z}(s).
    \end{align}
    
    For $r \geq s$ and  $ z \in K , $  
    we define:
    \begin{align}
        h^{z}(s,r) := z^{r-s} S^{z}(r-s) V(r) - z_{0}^{r-s} S^{z_{0}}(r-s) V(r).
    \end{align}
    We calculate 
\begin{align}
    z^{r-s} S^{z}(r-s) - z_{0}^{r-s} S^{z_{0}}(r-s) = \frac{1}{z - z^{-1}}  (z^{2 (r-s)} - 1 ) - \frac{1}{z_0 - z_0^{-1}}  (z_0^{2 (r-s)} - 1 ) 
    \end{align}
Remark \ref{IQ} yields:
   \begin{align}
        \| h^{z}(s,r) \| \leq C(K) |z - z_{0}|^{\min(\rho, 1)} (1 + (r-s))^{\min(\rho, 1)} \| V(r) \|.
    \end{align}
    This
     implies:
    \begin{align}
        \frac{\| h^{z}(s,r) \|}{(1+|s|)^{\min(\rho, 1 )}} \leq C(K ) |z - z_{0}|^{\min(\rho, 1)} (1 + |r|)^{\min(\rho, 1 )} \| V(r) \|,
    \end{align}
    since:
    \begin{align}
        \frac{1 + (r-s)}{1 + |s|} \leq 1 + |r|, \quad r \geq s.
    \end{align}
    From \eqref{f} and Lemma \ref{LAP1}, we conclude:
    \begin{align}\label{d}
        \sup_{s \in \mathbb{Z}} \| f^{z}(s) \| \leq C(K, \| V \|_0 ) |z - z_0|^{\min(\rho, 1 )} \| V \|_{\min(\rho, 1 )}
    \end{align}  (recall the definition of $\|\hspace{0.3cm}\|_{\min(\rho, 1)}$, \eqref{normV1}).
    
    On the other hand, using \eqref{ex2},  we calculate
    \begin{align}
        \frac{(1+|r|)^{\min(\rho, 1 )}}{(1+|s|)^{\min(\rho, 1 )}} \| z_{0}^{r-s} S^{z_{0}}(r-s) V(r) \| \leq C(K) (1 + |r|)^{\min(\rho, 1 )} \| V(r) \|,
    \end{align}
    for some constant $C(K)$. Thus:
    \begin{align}\label{e}
        \| g^{z}(s) \| \leq  \sum_{r=s+1}^{\infty} C(K)(1 + |r|)^{\min(\rho, 1 )} \| V(r) \| \left\| \frac{z^{-r} u_{+}^{z}(r) - z_{0}^{-r} u_{+}^{z_{0}}(r)}{(1+|r|)^{\min(\rho, 1 )}} \right\|.
    \end{align}
    
    Finally, combining \eqref{c}, \eqref{e}, and \eqref{d}, we obtain:
    \begin{align}
        &\left\| (1+|s|)^{-\min(\rho, 1 )} \left( z^{-s} u_{+}^{z}(s) - z_{0}^{-s} u_{+}^{z_{0}}(s) \right) \right\| \\
        \notag &\leq C(K, \| V \|_0 ) |z - z_0|^{\min(\rho, 1 )} \| V \|_{\min(\rho, 1 )} +& \\\notag  &\sum_{r=s+1}^{\infty} C(K)(1 + |r|)^{\min(\rho, 1 )} \| V(r) \| \left\| \frac{z^{-r} u_{+}^{z}(r) - z_{0}^{-r} u_{+}^{z_{0}}(r)}{(1+|r|)^{\min(\rho, 1 )}} \right\|.
    \end{align}
    Applying Gronwall's Lemma, we conclude:
    \begin{align}
        &\left\| (1+|s|)^{-\min(\rho, 1 )} \left( z^{-s} u_{+}^{z}(s) - z_{0}^{-s} u_{+}^{z_{0}}(s) \right) \right\| \leq \\ &\notag C(K, \| V \|_0  ) |z - z_0|^{\min(\rho, 1 )} \| V \|_{\min(\rho, 1 )} \exp \left( C(K)   \| V  \|_{\min(\rho, 1 ) } \right),
    \end{align}
    which completes the proof of \eqref{z} for $ u_+ $, after noticing that $  \frac{(1 + |s|)^{\min(\rho, 1 )}}{   (1 + |s|)^{\rho }} \leq 1   $. The proof of the other cases  is similar.
\end{proof}

\begin{propo}\label{LAP2}
   Let $K^{\pm}\subseteq  \overline{\mathbb{C}^{\pm}}\setminus ( \sigma_p (H)\cup \{-2,2 \} )$ be a compact set. 
   There is a constant $C(K^{\pm})$ independent of $E$, $s$ and $r$, such that:
   \begin{equation}\label{i2}
       \|[R^{\pm}_{H}(E)]_{s,r}\|\leq C(K^\pm).
   \end{equation}
\end{propo}
\begin{proof} We prove the statement for the plus sign.   Note that  that $r_+(K^+)$ is compact. Moreover,  $W(u_-^{\overline{z}^{-1}}, u_+^{z})$ is continuous and invertible for $ z \in  r_+(K^+)  $ (see Theorems \ref{Winvertible} and \ref{voltinK}, Lemma \ref{rpm}, \eqref{r_+}). It is, therefore, uniformly bounded on 
$ r_+(K^+) .$  Moreover,  $z^{-s}u_+^{z}(s)$ and $z^ r u_-^{\overline{z}^{-1}}(r)^{*}$ are uniformly bounded (as matrix multiplication operators) on $ r_+(K^+) $ with respect to $z,  r$ and $s$, see Lemma \ref{LAP1}. Then we have that 
\begin{align}
    |z^{s-r}|\left\|z^{-s}u_+^{z}(s)W(u_-^{\overline{z}^{-1}}, u_+^{z})^{-1}z^{r}u_-^{\overline{z}^{-1}}(r)^{*}\right\|\end{align}
    is uniformly bounded for $z \in r_+(K^+)$ and $  s \geq r . $ Now we use  \eqref{eq:Green} and \eqref{eq:Green1.1})  and deduce that 
   there is a constant $C(K^+)$ independent of $E$, $s$ and $r$ for $s\geq r$, such that:
   \begin{equation}
       \|[R^{+}_{H}(E)]_{s,r}\|\leq C(K^+),
   \end{equation}
   for every $ E \in K^+  .$ 
For the  case $ s \leq r$ we obtain similar bounds using the second line of the right hand side of \eqref{eq:Green}.  The analysis for $  K^- $ and $R_{H}(E - i 0)$ is performed similarly.

\end{proof}
 
\begin{obs}\label{resobon}
By Proposition \ref{LAP2}, if $E \in (-2,2)$, then the operator 
$
    R_{H}(E\pm i0): \ell^{1}(\mathbb{Z}, \mathcal{M}_{L}) \longrightarrow \ell^{\infty}(\mathbb{Z}, \mathcal{M}_{L})
$ 
defined by 
\begin{align}
    (R_{H}(E\pm i0)u)(s) = \sum_{r\in \mathbb{Z}} [R_{H}(E\pm i0)]_{s,r} u(r)
\end{align}
is well-defined (see Remak \ref{COP}). Consequently, for $u, v \in \ell^{1}(\mathbb{Z}, \mathcal{M}_{L})$, the product
\begin{align}
    \langle R_{H}(E \pm i0)u, v \rangle _{\ell ^{\infty} - \ell^1}:= \sum_{s \in \mathbb{Z}} \langle (R_{H}(E \pm i0)u)(s), v(s) \rangle _{\mathcal{M}_L}
\end{align}
is well-defined.
\end{obs}
\begin{teo}
    For all \(u, v \in \ell^{1}(\mathbb{Z}, \mathcal{M}_{L})\) and for every \(E_{0} \in (-2,2)\), we have:
    \begin{align}
       \lim_{E \to E_{0},\, \pm \Im E > 0} \langle R_{H}(E)u, v \rangle _{\ell ^{2}}= \langle R_{H}(E_{0}\pm i0)u, v \rangle _{\ell ^{\infty} - \ell^1}.
    \end{align}
\end{teo}

\begin{proof}
    Using \eqref{integraloperator}, we obtain:
    \begin{align}
        \langle (R_{H}(E) - R_{H}(E_{0}\pm i0))u, v \rangle _{\ell ^{\infty} - \ell^1} 
        &\notag= \sum_{s \in \mathbb{Z}} \langle (R_{H}(E) - R_{H}(E_{0}\pm i0))u(s), v(s) \rangle _{\mathcal{M}_L} \\
        &\notag= \sum_{s \in \mathbb{Z}} \left\langle \sum_{r \in \mathbb{Z}} [R_{H}(E) - R_{H}(E_{0}\pm i0)]_{s,r} \, u(r),\; v(s) \right\rangle _{\mathcal{M}_L}\\
        &= \sum_{s \in \mathbb{Z}} \sum_{r \in \mathbb{Z}} \left\langle [R_{H}(E) - R_{H}(E_{0}\pm i0)]_{s,r} \, u(r),\; v(s) \right\rangle _{\mathcal{M}_L}.
    \end{align}
    By the Cauchy-Schwarz inequality, it follows that:
    \begin{align}\label{LAP3}\notag
        \left| \langle (R_{H}(E) - R_{H}(E_{0}\pm  i0))u, v \rangle _{\ell ^{\infty} - \ell^1}\right| 
        &\leq \\  &\sum_{s \in \mathbb{Z}} \sum_{r \in \mathbb{Z}} 
        \left\| [R_{H}(E) - R_{H}(E_{0}\pm i0)]_{s,r} \right\| \,
        \| u(r) \|_{\mathcal{M}_{L}} \, \| v(s) \|_{\mathcal{M}_{L}}.
    \end{align}
    
    By Proposition \ref{LAP2}, there exists a constant $C$ such that $\left\| [R_{H}(E) - R_{H}(E_{0}+i0)]_{s,r} \right\| \leq C$ in some compact neighborhood of $E_0$. Therefore, using Lebesgue's Dominated Convergence Theorem and Theorem \ref{Green2.1}, it follows from \eqref{LAP3} that:
    \begin{align}
       \lim_{E \to E_{0},\, \pm \Im E > 0}  \left| \langle (R_{H}(E) - R_{H}(E_{0}\pm i0))u, v \rangle _{\ell ^{\infty} - \ell^1}\right| = 0.
    \end{align}
\end{proof}

\begin{defi}\label{sigma}
 For \(\alpha \in\mathbb{R}^{+}\) we define the operator 
$T_{-\alpha }:\ell ^{2}(\mathbb{Z},\mathcal{M}_L)\to \ell ^{2}(\mathbb{Z},\mathcal{M}_L),$
as
\begin{align}
   ( T_{-\alpha }u)(s)=(1+|s|)^{-\alpha}u(s)
\end{align}
which is a bounded operator for $\alpha \geq 0$. 
\end{defi}

\begin{teo}\label{teoaco} For every \(\alpha > \frac{1}{2}\) and \(E\in(-2,2)\), 
\begin{align}\label{i4}
    T_{-\alpha}R(E\pm i0)T_{-\alpha }\in\mathcal{B}(\ell ^{2}(\mathbb{Z},\mathcal{M}_L)).
\end{align}
Moreover, 
\begin{equation} \label{z3}\|T_{-\alpha }R(E\pm i0)T_{-\alpha }\|\leq C(\alpha )\sup_{s,r\in\mathbb{Z}}\|[R_H(E\pm i0)]_{s,r}\|.\end{equation} where $C(\alpha )$ is a constant which depends of $\alpha .$
\end{teo}
\begin{proof}
If \(u\in \ell ^{2}(\mathbb{Z},\mathcal{M}_L)\), the Cauchy-Schwarz inequality implies that
\begin{align}\label{z1}
\sum_{r\in\mathbb{Z}}(1+|r|)^{-\alpha }\|[R_H(E\pm i0)]_{s,r}u(r)\|_{\mathcal{M}_L}\leq \left(\sum_{r\in\mathbb{Z}}\frac{1}{(1+|r|)^{2\alpha }}\right)^{\frac{1}{2}}\sup_{s,r\in\mathbb{Z}}\|[R_H(E\pm i0)]_{s,r}\|\|u\|_{\ell ^{2}}.
\end{align}

Using \eqref{z1}, we have:
\begin{align}\label{z2}
&\sum_{s\in\mathbb{Z}}(1+|s|)^{-2\alpha }\left(\sum_{r\in\mathbb{Z}}(1+|r|)^{-\alpha }\|[R_H(E\pm i0)]_{s,r}u(r)\|_{\mathcal{M}_L}\right)^{2} \\&\notag\leq \left(\sum_{r\in\mathbb{Z}}\frac{1}{(1+|r|)^{2\alpha }}\right)^{2}\sup_{s,r\in\mathbb{Z}}\|[R_H(E\pm i0)]_{s,r}\|^{2}\|u\|_{\ell ^{2}}^{2}.
\end{align}
Equation \eqref{z2} implies that \(T_{-\alpha }R(E\pm i0)T_{-\alpha }\in\mathcal{B}(\ell ^{2}(\mathbb{Z},\mathcal{M}_L))\) and \eqref{z3}, with 
$  C(\alpha)=\left(\sum_{r\in\mathbb{Z}}\frac{1}{(1+|r|)^{2\alpha }}\right). $

\color{black}

\end{proof}
\color{black}
\begin{lema}\label{y11}
   Let $K \subset \overline{ \mathbb{D}} \setminus  \{0, -1,  1\} $ be a compact set. 
  Suppose that $ \|  V \|
  _{\rho} < \infty $. Then for $\rho > 0$ and $z \in K$, the following estimates hold:     
     
     \begin{align}\label{i5}
         \|W(u_-^{\overline{z}^{-1}}, u_+^{z})-W(u_-^{z_0}, u_{+}^{z_0})\|\leq C(K, \rho, \|  V\| _{\min(\rho, 1 )})   | z- z_0|^{\min(\rho, 1)} .
     \end{align}
\end{lema}
\begin{proof}
This is a direct consequence of  Lemma \ref{y10} and the definition of the Wronskian (which does not depend on $n$, see Definition \ref{wronskidefi} and Remark \ref{ND} ).


\end{proof}

\begin{lema} \label{y12} Suppose that  $\|  V \|
  _{\rho} < \infty .$ 
   Let $K^{\pm}\subseteq  \overline{\mathbb{C}^{\pm}}\setminus ( \sigma_p (H)\cup \{-2,2 \} )$ be a compact set. 
   Then, there is a constant $ C(K^{\pm}, \rho, \|  V\|_{\rho})$ such that:
    \begin{align}\label{i6}
    \frac{\|[R^{\pm}_{H}(E)]_{s,r}-[R^{\pm}_{H}(E_0)]_{s,r}\|}{(1+|r|)^{\rho }(1+|s|)^{\rho }}\leq C(K^{\pm}, \rho, \|  V \|_{\min(\rho, 1)})|z_0 -z|^{\min(\rho, 1) }  ,
\end{align}   $z= r^\pm(E),  z_0 = r^{\pm}(E_0),   \: E, E_0 \in K^{\pm}.$
\end{lema}
\begin{proof} First we assume that  $ s\geq r $ and take the plus sign. Note that (see  \eqref{eq:Green},  \eqref{eq:Green1.1})
    \begin{align}
     &  -i \Big (  [R^+_{H}(E)]_{s,r}-[R^+_{H}(E_0)]_{s,r}   \Big ) =\notag
     (z_0 ^{s-r} -z^{s-r})z_0 ^{-s}u_{+}^{z_0}(s)W(u_-^{\overline{z_0}^{-1}} ,u_{+}^{z_0})^{-1}z_0 ^{r}u_-^{\overline{z_0}^{-1}}(r)^{*}\notag\\
     &+z^{s-r}(z_0^{-s}u_{+}^{z_0}(s)-z^{-s}u_+^{z}(s))W(u_-^{\overline{z_0}^{-1}} ,u_{+}^{z_0})^{-1}z_0 ^{r}u_-^{\overline{z_0}^{-1}}(r)^{*}\notag\\
     &+z^{s-r}z^{-s}u_+^{z}(s)(W(u_-^{\overline{z_0}^{-1}} ,u_{+}^{z_0})^{-1}-W(u_-^{\overline{z}^{-1}},u_+^{z})^{-1})z_0^{r}u_-^{\overline{z_0}^{-1}}(r)^{*}\notag\\
     &+z^{s-r}z^{-s}u_+^{z}(s)W(u_-^{\overline{z}^{-1}},u_+^{z})^{-1}(z_0^{r}u_-^{z_0}(r)^{*}-z^{r}u_-^{\overline{z}^{-1}}(r)^{*}) .
 \end{align}

The first term can be estimated using Lemma \ref{LAP1}, Remark \ref{IQ} and the fact that $(1+|s|)(1+|r|)\geq (1+|s-r|)$. The second and fourth terms can be estimated using Lemma \ref{y10} and Lemma \ref{LAP1}. The third term is a consequence of Lemma \ref{LAP1} and Lemma \ref{y11}. 
The case $ s <  r $ can be similarly handled, the proof for the minus sign is analogous. 
\end{proof}

\begin{teo}[Limiting Absorption Principle]\label{LAPmain}  Suppose that $\|  V \|
  _{\rho} < \infty $ and $1/2<\alpha $,  with  $ \rho + 1/2 <  \alpha   $.  
 For every  compact set $ K^{\pm} \subset  \overline{\mathbb{C}}^\pm \setminus (\sigma_p(H ) \cup \{ -2, 2\}  )$  the next estimate holds: 
        \begin{align}\label{Ger2}
            \|T_{-\alpha }R^{\pm}_{H}(E)T_{-\alpha }-T_{-\alpha }R^{\pm}_{H}(E_0 )T_{-\alpha }\|\leq C(K^{\pm}, \alpha, \rho, \| V \|_{\min(\rho, 1)}) |r_\pm(E)-r_\pm (E_0)|^{   \rho },
        \end{align}
   for all  $ E , E_0\in K^{\pm}$. 
\end{teo}
\begin{proof}

We take the plus sign (taking the minus sign does not change the argument). 
Let \( u \in \ell ^2(\mathbb{Z}, \mathcal{M}_L) \). Using Cauchy-Schwarz inequality and  Lemma \ref{y12} we obtain that

\begin{align}\label{y13}&\notag\sum_{r\in\mathbb{Z}}(1+|r|)^{-\alpha}(1+|s|)^{-\rho}\|([R_{H}(E)]_{s,r}-[R_{H}(E_0)]_{s,r})u(r)\|_{\mathcal{M}_L}\\\notag
&\leq \sum_{r\in\mathbb{Z}}(1+|r|)^{-\alpha + \rho  }\frac{\|[R_{H}(E)]_{s,r}-[R_{H}(E_0)]_{s,r}\|}{(1+|r|)^{ \rho}(1+|s|)^{\rho}}\|u(r)\|_{\mathcal{M}_L}\\
&\leq \left(\sum _{r\in \mathbb{Z}}(1+|r|)^{2(-\alpha + \rho) }\right)^{\frac{1}{2}}C(K^{\pm}, \rho, \| V\|_{\min(\rho, 1)})|r_+(E)-r_+(E_0)|^{\min(\rho, 1)}\|u\|_{\ell ^ 2}.
\end{align}

Therefore, using \eqref{y13} we have

\begin{align}&\notag \|(T_{-\alpha}R_{H}(E)T_{-\alpha}-T_{-\alpha}R_{H}(E_0)T_{-\alpha})u\|^{2}_{\ell ^{2}}\\\notag
&\leq\sum_{s\in\mathbb{Z}}(1+|s|)^{-2\alpha}\left(\sum_{r\in\mathbb{Z}}(1+|r|)^{-\alpha}\|([R_{H}(E)]_{s,r}-[R_{H}(E_0)]_{s,r})u(r)\|\right)^{2}\\\notag
&=\sum_{s\in\mathbb{Z}}(1+|s|)^{2(- \alpha +  \rho) }\left(\sum_{r\in\mathbb{Z}}(1+|r|)^{-\alpha}(1+|s|)^{-\rho}\|([R_{H}(E)]_{s,r}-[R_{H}(E_0)]_{s,r})u(r)\|\right)^{2}\\
&\leq C(K^{\pm }, \rho,\alpha, \| V \|_{ \min(\rho, 1)  } ) |r_+(E)-r_+(E_0)|^{2 \min(\rho, 1) }\|u\|_{\ell ^{2}}^{2}.
\end{align}

\end{proof}

\color{black}

 \section{Dispersive estimates}\label{dispersiveesimates}

\color{black}

The explicit formula for the resolvent kernel allows for an explicit derivation of the spectral measure of $H$. This result is summarized in the following lemma:
 
\begin{lema}
  Let \(\mathcal{E}_{H}(\omega)\) be the spectral measure of \(H\). 
For each \( z \in \mathbb{S}^1 \setminus \{-1, 1\} \), we define:
\begin{equation}\label{qwe}
   f_{s,r}(z)=
 \left\{ \begin{array}{lcc}
          \tilde{u}_{+}^{z}(s)T_{+}^{z} \tilde{u}_{-}^{z}(r)^{*}&   s\geq r \\
  \tilde{u}_- ^{z^{-1}}(s)T_-^{z}\tilde{u}_+^{z^{-1}}(r)^{*}&  s < r
             \end{array} ,
   \right.
   \end{equation}     
   where $T_\pm ^{z}$ is defined in  \eqref{eq:TR1} and $\tilde{u}^{z^{\pm 1}}$ are the functions defined in Definition \ref{utilde}.    
      
    For $-2<a<b<2$, the following formula holds true:
    \begin{align}
       [\mathcal{E}_{H}(a,b) ]_{s,r} =  \frac{1}{2 \pi i}\int_{a}^{b}  \frac{ 1 }{ z - z^{-1} }  \Big ( z^{|s-r|} f_{s, r}(z)  + z^{-|s-r|} f_{s, r}(z^{-1}) \Big )  dE ,
    \end{align}
$  z = r_+(E) .$
\end{lema}
 
\begin{proof}
\label{teo:stone}

By Lemma \ref{LAP2}, the terms $[R_{H}(E)]_{s,r}$ are uniformly bounded with respect to $z$ in sets $ \{ E\in \mathbb{C}: a \leq \Re(E)\leq b, \hspace{0.2cm}  0 \leq   |\Im(E)|<\varepsilon \}$ for $-2<a<b<2$ and $\varepsilon >0$. By Stone's formula (see \cite{MR1009163}, p. 920, Theorem 1) and the dominated convergence theorem, we have, for all $A, B\in \mathcal{M}_L$, that

\begin{align}
 &\left<[ \mathcal{E}_{H}((a,b))  ]_{s,r}A,B\right>_{\mathcal{M}_{L}}=\left<\pi _s\mathcal{E}_{H}((a,b))\pi _r^{*}A,B\right>_{\mathcal{M}_{L}}= \left<\mathcal{E}_{H}((a,b))\pi _r^{*}A, \pi _s^{*}B\right>_{\ell ^{2}} \\\notag &=\lim _{\delta \to 0^{+}}   \lim _{\varepsilon \to 0^{+}}  \frac{1}{2\pi i}\int _{a+\delta }^{b-\delta } \left< \Big ( R_{H}(E +i\varepsilon )  -R_{H}(E -i\varepsilon )  \Big )\pi _r^{*}A, \pi _s^{*}B\right>_{\ell ^{2}} dE.  
 \\ \notag & =    \frac{1}{2\pi i}\int _{a }^{b } \left< \Big ( R_{H}(E +i 0 )  -R_{H}(E -i 0 )  \Big )\pi _r^{*}A, \pi _s^{*}B\right>_{\ell ^{2}} dE. 
  \end{align}
The desired result follows from \eqref{eq:Green1} and \eqref{eq:Green2} (here, $  z = r_+(E) $) :
\begin{align}
\Big [ R_{H}(E +i 0 )  -R_{H}(E -i 0 ) \Big ]_{s, r}= 
    \frac{ 1 }{ z - z^{-1} }  \Big ( z^{|s-r|} f_{s, r}(z)  + z^{-|s-r|} f_{s, r}(z^{-1}) \Big )  .
    \end{align}


\end{proof}

We recall that \( P_{ac} = \mathcal{E}_{H}([-2,2]) \), and \( \tilde{\mathcal{E}}_{H} := P_{ac} \mathcal{E}_{H} \).  
Using the fact that \(-2\) and \(2\) are not eigenvalues of \(H\) (see Theorem 4 in \cite{MR4760547}), we obtain the following result:

   \begin{cor}\label{caro} The next equality holds true:
    \begin{align}
    \label{eq:derivative}
    \\ &\notag \frac{d \left< \tilde{ \mathcal{E}}_{H}\pi _r^{*}A, \pi _s^{*}B\right>_{\ell ^{2}}}{dE} = \begin{cases}  \notag \frac{1}{2\pi i} \left<(  \frac{ 1 }{ z - z^{-1} }  \Big ( z^{|s-r|} f_{s, r}(z)  + z^{-|s-r|} f_{s, r}(z^{-1}) \Big ) )A, B\right>_{\mathcal{M}_L} , \:  & E \in (-2,2) , \\ 
    0 ,  & \text{otherwise}\end{cases}
    \end{align}
  for all $A,B\in \mathcal{M}_L$ and $ z = r_+(E) .$
\end{cor}

Corollary \ref{caro} and the fact that the essential spectrum of \( H \) is \([-2,2]\) verify that the projection onto the absolutely continuous subspace of \( H \) is \( P_{ac} \).

\begin{lema}
\label{evolutionkernel}
  The integration kernel for $e^{-itH}P_{ac}$ is given by
    \begin{align}\label{evolutionkernel}
         [ e^{-itH}P_{ac}]_{s,r}= \int_{-\pi}^{\pi} e^{-it2\cos k -i|s-r|k} f_{s,r}(e^{-ik}) dk.
    \end{align} 
    
\end{lema}
\begin{proof}
By the spectral Theorem and (\ref{eq:derivative}), we have  that 
\begin{align}\label{timeevolution1}
\\ & \notag [  e^{-itH}P_{ac}]_{s,r} = \frac{1}{2 \pi i}  \int_{-2}^2 e^{-it E } \frac{ 1 }{ z - z^{-1} }  \Big ( z^{|s-r|} f_{s, r}(z)  + z^{-|s-r|} f_{s, r}(z^{-1}) \Big ) dE
\end{align} where $z=r_+(E)$.
By \eqref{r_+}, if $E\in (-2, 2),$   $r_+(E)= e^{-ik } = z$ for $E = 2\cos (k), k\in [0,\pi ]$, then, equation \eqref{timeevolution1} implies:

\begin{align}
    [e^{-itH}P_{ac}]_{s,r} &= -\frac{1}{2\pi i}\int _{ 0}^{\pi} \frac{ e^{-it2\cos k}}{-2i\sin k}\left[ e^{- i| s-r |k }f_{s,r}(e^{-ik }) + e^{ i|s-r|k }f_{s,r}(e^{ik })\right](-2\sin k )dk \\
         & \notag =  \frac{1}{2\pi} \left( \int_{0}^{\pi } e^{-it2\cos k - i|s-r|k} f_{s,r}(e^{-ik}) dk + \int_{-\pi}^{0} e^{-it2\cos k - i|s-r|k} f_{s,r}(e^{-ik}) dk \right)  \\ 
         & \notag =  \frac{1}{2\pi}\int_{-\pi}^{\pi} e^{-it2\cos k - i|s-r|k} f_{s,r}(e^{-ik}) dk.
\end{align}

\end{proof}
 
 \begin{defi}\label{generic}
     The potential $V$ is called generic provided that $W(u_-^{z}, u_+^{z})$ and $W(u_+^{z^{-1}}, u_-^{z^{-1}})$ are invertible for every $z \in \mathbb{S}^{1}$.
 \end{defi}
\color{black}

   \begin{lema} \label{y3}If  $V$ is generic (in the sense of Definition \ref{generic}), then the functions $\mathbb{S}^{1}\ni z\mapsto T_{\pm}^{z}$ and $\mathbb{S}^{1}\ni z\mapsto R_{\pm}^{z}$ are elements of the Wiener algebra $  \mathcal{A}_{\mathcal{M}_L} $.
   \end{lema}
   \begin{proof}
       By Proposition \ref{Abound}, for every $ s \in \mathbb{Z} $, the functions $  \mathbb{S}^{1}\ni z \mapsto u^{z^{\pm 1}}_{\pm}(s)$
(see Definition \ref{jost}) belong to the Wiener Algebra $   \mathcal{A}_{\mathcal{M}_L}  $ (see \eqref{eq:WM}). Then, due to the algebra structure of \(  \mathcal{A}_{\mathcal{M}_L}  \) and the Definition of Wronskian (Def. \ref{wronskidefi}), it follows that
\begin{align}\label{w}
    W(u_-^{z}, u_+^{z})\in  \mathcal{A}_{\mathcal{M}_L} , \hspace{1cm}
   W(u_- ^{z^{-1}}, u_+ ^{z})\in  \mathcal{A}_{\mathcal{M}_L} .
\end{align}
Since \( W(u_-^{z}, u_+^{z})\in  \mathcal{A}_{\mathcal{M}_L}  \), if it is invertible for all \( z\in \mathbb{S}^{1} \), then by Theorem \ref{thm:wienerG} it follows that  
\begin{align}\label{yy}
W(u_-^{z}, u_+^{z})^{-1}\in \mathcal{A}_{\mathcal{M}_L} .
\end{align}
Therefore, from \eqref{nu2}, \eqref{MN}, \eqref{eq:TR1},  \eqref{yy}, and the 
 fact that $\mathcal{A}_{\mathcal{M}_L}$ is an algebra, it follows from that 
\begin{align}
    T_{+}^{z}=(M_+^{z})^{-1}= -i (z - z^{-1}) W(u_-^{z}, u_+^{z})^{-1}\in  \mathcal{A}_{\mathcal{M}_L} .
\end{align}

On the other hand, from \eqref{eq:TR1}, \eqref{w}, and \eqref{yy} it follows that 
\begin{align}
    R_{+}^{z}=-N_{+}^{z} (M_+^{z})^{-1}= \nu^z W(u_- ^{z^{-1}}, u_+ ^{z}) (\nu^z)^{-1}  W(u_-^{z}, u_+^{z})^{-1}
    \in  \mathcal{A}_{\mathcal{M}_L} .
\end{align}

The proofs for the coefficients \( T_{-}^{z} \) and \( R_{-}^{z} \) are similar.
   \end{proof}
   
     We can estimate \eqref{evolutionkernel} using Theorem \ref{VanderWiener}. For this we must first prove that $f_{s,r}\in  \mathcal{A}_{\mathcal{M}_{L}}$ and it is uniformly bounded with respect to $s$ and $r$. This analysis is contained in the following lemma, which is a direct consequence of many estimates we present above.  For completeness and the convenience of the readers,  we include the proof (we follow the arguments in \cite{MR3433284}, where the scalar case is addressed).  
 
    \begin{lema} \label{ubound}
       There is a constant $ C $ such that 
      \begin{align}\sup_{s, r} \lVert f_{s, r} \rVert _{\mathcal{A}_{\mathcal{M}_L}}< C . \end{align}
    \end{lema}
    \begin{proof} 
We only prove the statement for $ r \leq s $. In the circumstance that $  r > s $, the arguments are analogous. We study separately the cases   
 $    r \leq 0 \leq s $ , 
     $  r \leq s \leq 0 $ , 
     $  0 \leq r \leq s  .  $

{\textbf Case 1.     $ r \leq 0 \leq s $.} 
By (\ref{one}) and Lemma \ref{y3}, we have (for $ z \in \mathbb{S}^1 $) that
\begin{equation}
     \| f_{s,r} \|_{\mathcal{A}_{\mathcal{M}_L}} \leq \|\tilde{u}_+ ^{z}(s)\|_{\mathcal{A}_{\mathcal{M}_L}}\|T_+^{z}\|_ {\mathcal{A}_{\mathcal{M}_L}}\|\tilde{u}_-^{z}(r)^{*}\|_{\mathcal{A}_{\mathcal{M}_L}} 
\end{equation}
is uniformly bounded with respect to \( s \) and \( r \).

{\textbf Case 2.     $ r \leq s \leq 0  $.} By Definition \ref{utilde} and Equation (\ref{TReq}), we have
\begin{align}
    \tilde{u}_+^{z}(s) T_+^{z} 
    \notag = \tilde{u}_-^{z}(s) - z^{-2s} \tilde{u}_-^{z^{-1}}(s) R_+^{z},
\end{align}
and therefore
\begin{align}
    f_{s,r}(z) = \tilde{u}_+^{z}(s) T_+^{z} \tilde{u}_-^{z}(r)^* = \left( \tilde{u}_-^{z}(s) - z^{-2s} \tilde{u}_-^{z^{-1}}(s) R_+^{z}\right) \tilde{u}_-^{z}(r)^*.
\end{align} Using  that $\mathbb{S}^{1}\ni z\mapsto z^{-2s}$ has norm $L$ in the Wiener algebra (see \eqref{eq:WM}), by   (\ref{one}) and Lemma \ref{y3},  we have that
\begin{equation}
     \| f_{s,r} \|_{\mathcal{A}_{L}} \leq (\|\tilde{u}_-^{z}(s)\|_{\mathcal{A}_{\mathcal{M}_L}} + L \| \tilde{u}_-^{z^{-1}}(s) \|_{\mathcal{A}_{\mathcal{M}_L}}\|R_+^{z} \|_{\mathcal{A}_{\mathcal{M}_L}} )\| \tilde{u}_-^{z}(r)^*\|_{\mathcal{A}_{\mathcal{M}_L}}
\end{equation}
is uniformly bounded with respect to \( s \) and \( r \).

\textbf{Case 3. \( 0 \leq r \leq s \).} By Equation (\ref{TReq}), \eqref{Tproperty} , we have 
\begin{align}
    T_+^{z} u_-^{z}(r)^* = (T_-^{z^{-1}})^* u_-^{z}(r)^* = u_+^{z}(r)^* - (R_-^{z^{-1}})^* u_+^{z^{-1}}(r)^*,
\end{align}
and therefore
\begin{align}
    T_+^{z} \tilde{u}_-^{z}(r)^*  
     = z^{r} \left( z^{-r} \tilde{u}_+^{z}(r)^* - z^{r} (R_-^{z^{-1}})^* \tilde{u}_+^{z^{-1}}(r)^* \right) 
     = \tilde{u}_+^{z}(r)^* - z^{2r} (R_-^{z^{-1}})^* \tilde{u}_+^{z^{-1}}(r)^*.
\end{align}
  Then
  \begin{align}
      f_{s,r}(z)=\tilde{u}_{+}^{z}(s)T_{+}^{z}\tilde{u}_{-}^{z}(r)^{*}=\tilde{u}_{+}^{z}(s)\big(\tilde{u}_+^{z}(r)^* - z^{2r} (R_-^{z^{-1}})^* \tilde{u}_+^{z^{-1}}(r)^*\big).
  \end{align}
  By equation \eqref{one} and Lemma \ref{y3}, and using that $\mathbb{S}^{1}\ni z\mapsto z^{2r}$ has norm $L$ in the Wiener algebra (see \eqref{eq:WM}),  we have that  
  \begin{equation}
      \|f_{s,r}\|_{\mathcal{A}_{\mathcal{M}_L}}\leq \|\tilde{u}_{+}^{z}(s)\|_{\mathcal{A}_{\mathcal{M}_L}} \big(\|\tilde{u}_{+}^{z}(r)^{*}\|_{\mathcal{A}_{\mathcal{M}_L}}+ L \|(R_{-}^{z^{-1}})^{*}\|_{\mathcal{A}_{\mathcal{M}_L}} \|\tilde{u}_{+}^{z^{-1}]}(r)^{*}\|_{\mathcal{A}_{\mathcal{M}_L}}\big)
  \end{equation}
  is uniformly bounded respectively to $s$ and $r$.

\end{proof}

\begin{teo}\label{EST}
    The following estimate holds true: 
    $$ \|e^{-itH}P_{ac}\|_{\ell^{1} \to \ell^{\infty}} = O(t^{-\frac{1}{3}}), \quad t \to \infty , $$
    where we recall that \( P_{ac} \) denotes the projection in \( \ell^{2}(\mathbb{Z}, \mathcal{M}_{L}) \) onto the absolutely continuous subspace of \( H \).
\end{teo}

\begin{proof}
    By  \eqref{evolutionkernel}, we have that
    \begin{align}
         [e^{-itH}P_{ac}]_{s,r}=\int_{-\pi}^{\pi} e^{-it\phi _a (k )} f_{s,r}(e^{-ik}) dk, 
    \end{align} 
  where  
    $\phi _a (k ):=2\cos k +ak $ for $a= \frac{|s-r|}{t} $. We  split  $[-\pi, \pi]$ as bellow 
    $$[-\pi ,\pi ]=\left[-\pi, -\dfrac{3\pi}{4}\right]\cup \left[-\frac{3\pi }{4}, -\frac{\pi}{4}\right]\cup \left[-\frac{\pi}{4},\frac{\pi}{4}\right]\cup \left[\frac{\pi }{4}, \dfrac{3\pi}{4}\right]\cup \left[\frac{3\pi }{4}, \pi \right].$$ 
    Notice that in the intervals  $\left[-\frac{3\pi }{4}, -\frac{\pi}{4}\right]$ and $\left[\frac{\pi }{4},\frac{3\pi}{4}\right]$, $|\phi _{a} ^{\prime \prime \prime }(k )|=|2\sin k |\geq 1$, and in the intervals $[-\pi, -\frac{3\pi }{4}]$, $\left[\frac{3\pi }{4}, \pi \right]$ and $\left[-\frac{\pi }{4}, \frac{\pi}{4}\right]$, $|\phi _{a} ^{\prime \prime }(k )|=|2\cos k |\geq 1$ . Then, we use Theorem \ref{VanderWiener}  for $l=2$ and $l=3,$ and  the fact $t^{-\frac{1}{2}}\leq t^{-\frac{1}{3}}$ for $t>1,$  to conclude that
    \begin{equation}
        \|[e^{-itH}P_{ac}]_{s,r}\| \leq C  t^{-\frac{1}{3}}
    \end{equation}
    for some constant \( C \),  since by the Lemma \ref{ubound} \( \{\|f_{s,r}(z)\|_{\mathcal{A}_{\mathcal{M}_{L}}}\} \)  is uniformly bounded respect to \( s, r \in \mathbb{Z} .\)
    Therefore, if \( u \in \ell^{1}(\mathbb{Z}, \mathcal{M}_{L}) \), it follows that for all \( s \in \mathbb{Z} \),
    $$ \|(e^{-itH}P_{ac} u)(s)\|_{\mathcal{M}_{L}} \leq \sum_{r \in \mathbb{Z}} \|[e^{-itH}]_{s,r} u(r)\|_{\mathcal{M}_{L}} \leq C t^{-\frac{1}{3}} \left( \sum_{r \in \mathbb{Z}} \|u(r)\|_{\mathcal{M}_L} \right) = C t^{-\frac{1}{3}} \|u\|_{\ell ^{1}}. $$
\end{proof}

\color{black}

\section*{Acknowledgement}

This work was supported  by CONACYT, FORDECYT-PRONACES 429825/2020 and PAPIIT-DGAPA-UNAM  IN114925. Furthermore, G. Franco C\'ordova received funding from the DAAD. M. Ballesteros is a fellow of SNII, SECIHTI.

\bibliographystyle{abbrv} 
\bibliography{research_bibliography}
\end{document}